\documentclass[a4paper, 11pt, headings = small, abstract]{scrartcl}
\usepackage[english]{babel}

\usepackage{amscd,amsmath,amsthm,array,hhline,mathrsfs, enumerate, booktabs,fancybox,calc,xcolor,graphicx,bbm,xspace,nicefrac,stmaryrd,url,arcs,stmaryrd}

\counterwithin{figure}{section}
\newcommand{\ao}{\widehat{\maxlevy}_T}

\usepackage[theoremfont]{newpxtext}
\usepackage[vvarbb, upint, bigdelims]{newpxmath}
\usepackage[scr=boondoxupr, scrscaled = 1.05, cal = euler, calscaled =1.05]{mathalpha}

\DeclareMathAlphabet{\mathsf}{OT1}{\sfdefault}{m}{n}
\SetMathAlphabet{\mathsf}{bold}{OT1}{\sfdefault}{b}{n}
\DeclareMathAlphabet{\mathfrak}{U}{jkpmia}{m}{it}
\SetMathAlphabet{\mathfrak}{bold}{U}{jkpmia}{bx}{it}
\usepackage[utf8]{inputenc}

\usepackage{enumitem}
\usepackage{etoolbox}
\usepackage[normalem]{ulem}
\usepackage{mathtools}
\usepackage{geometry}
\usepackage{float}

\usepackage{bm}
\usepackage{tikz}
\usepackage[outdir =./]{epstopdf}
\usepackage[citestyle = numeric-comp,bibstyle = numeric, firstinits=true, natbib = true, backend = bibtex,maxbibnames=99, url = false, isbn = false]{biblatex}
\numberwithin{equation}{section}

\usepackage{xcolor}
\usepackage[colorlinks=true, allcolors=myteal]{hyperref}
\definecolor{grey_pers}{RGB}{69 90 100}
\definecolor{WIMgreen}{RGB}{60 134 132}
\definecolor{red_pers}{RGB}{204 37 41}
\definecolor{UMblue}{RGB}{4 47 86}
\definecolor{myteal}{RGB}{0 123 137}
\definecolor{dartmouthgreen}{rgb}{0.05, 0.5, 0.06}\definecolor{cobalt}{rgb}{0.0, 0.28, 0.67}\definecolor{coolblack}{rgb}{0.0, 0.18, 0.39}
\definecolor{glaucous}{rgb}{0.38, 0.51, 0.71}\definecolor{hooker\'sgreen}{rgb}{0.0, 0.44, 0.0}\definecolor{lemonchiffon}{rgb}{1.0, 0.98, 0.8}\definecolor{oucrimsonred}{rgb}{0.6, 0.0, 0.0}\definecolor{radicalred}{rgb}{1.0, 0.21, 0.37}\definecolor{raspberry}{rgb}{0.89, 0.04, 0.36}\definecolor{royalazure}{rgb}{0.0, 0.22, 0.66}
\definecolor{dex}{RGB}{138 18 34}

\theoremstyle{plain}
\newtheorem{theorem}{Theorem}[section]
\newtheorem{proposition}[theorem]{Proposition}
\newtheorem{lemma}[theorem]{Lemma}
\newtheorem{corollary}[theorem]{Corollary}

\theoremstyle{definition}

\theoremstyle{assumption}

\theoremstyle{remark}
\newtheorem{remark}[theorem]{Remark}

\setkomafont{sectioning}{\normalcolor\bfseries}
\setkomafont{descriptionlabel}{\normalcolor\bfseries}
\setkomafont{section}{\Large}
\setkomafont{subsection}{\large}
\setkomafont{subsubsection}{\large}
\setkomafont{paragraph}{\large}
\setkomafont{subparagraph}{\large}
\setkomafont{author}{\large}
\setkomafont{date}{\normalsize}

\def\coa{\bm{\vartheta}}
\def\cob{\bm{\varrho}}
\def\coc{\bm{\chi}}

\def\supp{\operatorname{supp}}

\def\TV{\operatorname{TV}}
\def\C{\mathfrak C}

\def\E{\mathbb{E}}
\def\G{\mathbb{G}}
\def\F{\mathbb{F}}

\def\N{\mathbb{N}}

\def\N{\mathbb{N}}

\def\R{\mathbb{R}}
\definecolor{darkred}{rgb}{0,0.6,0}
\def\X{\mathbf{X}}

\def\cX{\mathcal{X}}

\def\Var{\mathrm{Var}}

\newcommand{\cB}{\mathcal{B}}

\newcommand{\cF}{\mathcal{F}}

\newcommand{\ep}{\varepsilon}
\newcommand{\cO}{\mathcal{O}}

\newcommand{\PP}{\mathbb{P}}

\renewcommand{\subseteq}{\subset}

\renewcommand{\hat}{\widehat}
\newcommand{\e}{\mathrm{e}}
\renewcommand{\tilde}{\widetilde}%
\renewcommand{\d}{\mathop{}\!\mathrm{d} }
\newcommand{\lebesgue}{\boldsymbol{\lambda}}

\newcommand{\colow}{\xi}
\newcommand{\coup}{\theta}
\newcommand{\estcost}{\hat{C}_T(\colow,\coup)}
\newcommand{\estbound}{(\hat{\colow}_T,\hat{\coup}_T)}
\newcommand{\estboundcontrol}{(\tilde{\colow}_{s},\tilde{\coup}_{s})}
\newcommand{\maxlevy}{\coup}

\newcommand\indep{\protect\mathpalette{\protect\independenT}{\perp}}
\def\independenT#1#2{\mathrel{\rlap{$#1#2$}\mkern2mu{#1#2}}}
\newcommand{\overbar}[1]{\mkern 1.5mu\overline{\mkern-1.5mu#1\mkern-1.5mu}\mkern 1.5mu}
\newcommand{\underbars}[1]{\mkern 1.5mu\underline{\mkern-1.5mu#1\mkern-1.5mu}\mkern 1.5mu}

\def\Bd{B}
\def\Val{C^*}

\restylefloat{table}

\newcommand*\diff{\mathop{}\!\mathrm{d} }

\newcommand{\one}{\mathbf{1}}

\newcommand{\vertiii}[1]{{\left\vert\kern-0.25ex\left\vert\kern-0.25ex\left\vert #1
    \right\vert\kern-0.25ex\right\vert\kern-0.25ex\right\vert}}

\def\supp{\mathrm{supp}}

\let\originalleft\left
\let\originalright\right
\renewcommand{\left}{\mathopen{}\mathclose\bgroup\originalleft}
\renewcommand{\right}{\aftergroup\egroup\originalright}
\makeatletter
\def\@fnsymbol#1{\ensuremath{\ifcase#1\or \dagger \or * \or \ddagger\or
   \mathsection\or \mathparagraph\or \|\or **\or \dagger\dagger
   \or \ddagger\ddagger \else\@ctrerr\fi}}
\makeatother

\allowdisplaybreaks

\makeatother
\title{\fontsize{16}{19} \selectfont Learning to reflect}
\subtitle{A unifying approach for data-driven stochastic control strategies}

\author{S\"oren Christensen\thanks{Christian-Albrechts-Universit\"at Kiel, Mathematisches Seminar, Ludewig-Meyn-Str. 4, 24098 Kiel, Germany. \newline Email: \href{mailto:christensen@math.uni-kiel.de}{christensen@math.uni-kiel.de}} \and Claudia Strauch\thanks{Aarhus University, Department of Mathematics, Ny Munkegade 118, 8000 Aarhus C, Denmark. \newline Email: \href{mailto:strauch@math.au.dk}{strauch@math.au.dk}} \and Lukas Trottner\thanks{Universit\"at Mannheim, Institut f\"ur Mathematik, B6 26, 68159 Mannheim, Germany. \newline Email: \href{mailto:ltrottne@mail.uni-mannheim.de}{ltrottne@mail.uni-mannheim.de}}}
\date{\vspace{-5ex}}

\bibliography{levy_kontrolle}

\begin{document}

\maketitle
\begin{abstract}
Stochastic optimal control problems have a long tradition in applied probability, with the questions addressed being of high relevance in a multitude of fields. Even though theoretical solutions are well understood in many scenarios, their practicability suffers from the assumption of known dynamics of the underlying stochastic process, raising the statistical challenge of developing purely data-driven strategies. For the mathematically separated classes of continuous diffusion processes and L\'evy processes, we show that developing efficient  strategies for related singular stochastic control problems can essentially be reduced to finding rate-optimal estimators with respect to the $\sup$-norm risk of objects associated to the invariant distribution of ergodic processes which determine the theoretical solution of the control problem. 
From a statistical perspective, we exploit the exponential $\beta$-mixing property as the common factor of both scenarios to drive the convergence analysis, indicating that relying on general stability properties of Markov processes
is a sufficiently powerful and flexible approach to treat complex applications requiring  statistical methods. 
We show moreover that in the L\'evy case---even though per se jump processes are more difficult to handle both in statistics and control theory---a fully data-driven strategy with regret of significantly better order than in the diffusion case can be constructed. 

\end{abstract}


\vspace{.2cm}

\section{Introduction}\label{sec:intro}
From a purely mathematical point of view, the field of statistics of stochastic processes is very appealing as it lives from the combination of different techniques and findings from diverse mathematical areas, in particular statistics, probability theory or functional analysis. 
The fundamental motivation of this branch of statistics, however, results from concrete applications.
Thus, besides mathematical elegance and completeness, the developments and results in this area should always be tested in terms of their applicability. 

An important area in which stochastic processes (especially of diffusion-type) are used by default to account for random impacts is stochastic control theory. Whereas the theory itself is very well developed and offers concrete decision strategies for a variety of problems, these 
are usually based on the assumption that the decision maker has full knowledge of the dynamics of the underlying random process. In \cite{christensen20}, we have already presented an approach to overcome this constraint by means of nonparametric estimation methods and proposed a fully data-driven approach to solving a concrete impulse control problem.
In this paper, we are expanding the view and approaching the problem from a general perspective. Basic components for the data-based solution of a large class of stochastic control problems are
\begin{itemize}
\item[>>] the control of the $\sup$-norm risk for the estimation of certain (functionals of) characteristics of the random process, in particular
\item[>>] the derivation of upper bounds on the convergence rate.
\end{itemize}
In Section \ref{subsec:motivating_control}, we describe the nature of the control problems and how they naturally lead to associated nonparametric estimation problems. Based on this, in Section \ref{sec:2}, we briefly formulate our general statistical modelling framework. 

\subsection{The motivating control problems}\label{subsec:motivating_control}
The stochastic control problems we consider in this paper are---under the assumption that the decision maker has access to the underlying dynamics---classical, and variants are well-studied. They have in common that a decision maker controls a continuous-time process $\X$ on the real line, but the controls do not change continuously over time, but are of a singular type. More precisely, it turns out that the optimal strategies call for reflecting the underlying process at certain boundaries. These optimal boundaries can be found (semi-) explicitly as optimizers of certain (deterministic) auxiliary functions, based on the dynamics of the underlying uncontrolled process. In this paper, we consider the more realistic situation that the decision maker has to estimate the underlying dynamics while controlling the process. 
The main key for such a statistical treatment is that, for an underlying ergodic scalar diffusion $\X$, the corresponding auxiliary function can be described explicitly in terms of the invariant density, as detailed in\ Section \ref{sec:diff}. For estimating the optimizer in this case, the $\sup$-norm risk of invariant density estimators has to be studied. 
In Section \ref{sec:levy}, we then turn our attention to a control problem for underlying L\'evy processes. In this case, the auxiliary function is identified as a generator functional $\mathcal{A}_H \gamma$ of the ascending ladder height process $\mathbf{H}$ belonging to a L\'evy process $\X$. Again, a $\sup$-norm estimation procedure for such functionals has to be found.

A data-driven solution method for the control problems therefore naturally leads to the challenging statistical problem of setting set up a framework such that the seemingly different issues of $\sup$-norm estimation of the invariant density of an ergodic diffusion on the real axis and $\sup$-norm estimation of ladder height generator functionals $\mathcal{A}_H \gamma$ for a L\'evy process $\X$ can be integrated into.

\subsection{Nonparametric analysis: Controlling the $\mathbf{sup}$-norm risk of Markovian functionals}\label{sec:2}
The identification of an appropriate technical framework is a crucial issue for the statistical analysis of stochastic processes. 
Specific model choices such as scalar diffusion processes or multivariate reversible processes with continuous trajectories permit the application of particular technical tools (e.g., associated to diffusion local time or to the symmetry of the semigroup), but generally do not provide any information about the robustness of the used statistical methods beyond the chosen framework. 
In contrast, exponential $\beta$-mixing of general continuous-time Markov processes $\X = (X_t)_{t \geq 0}$ has been identified in \cite{dexheimer20} as a criterion which, on the one hand, is strong enough to serve as a central building block of a robust statistical analysis while, on the other hand, providing sufficient generality to allow to include an exhaustive list of Markov processes in the framework. 
Statistical properties of such processes can thus be studied based on fairly general results rooted in stability theory of Markov processes. We recall them at this point in order to apply them afterwards to estimation problems that are central for developing efficient data-driven stochastic control procedures. 

\paragraph{Basic notions and results}
Suppose that $\X = (X_t)_{t \geq 0}$ is a Borel right Markov process on some Borel state space $(\cX,\cB(\cX))$ with semigroup 
$$P_t(x,B) \coloneqq \PP^x(X_t \in B), \quad x \in\cX, B \in \cB(\cX),$$
and unique invariant distribution $\mu$, i.e., for any $B \in \cB(\cX)$ it holds that 
$$\PP^\mu(X_t \in B) = \int_{\cX} \PP^x(X_t \in B)\, \mu(\diff{x}) = \int_{\cX} P_t(x,B)\, \mu(\diff{x}) = \mu(B).$$
Denote, for $t \geq 0$, $\cF^0_t \coloneqq \sigma(X_s,0\leq s \leq t)$ and $\overbar{\cF}{}^0_t \coloneqq \sigma(X_s,s \geq t)$.  
We say that $\X$, started in some distribution $\eta$ on the state space, is exponentially $\beta$-mixing if there exist constants $\kappa,c_\kappa> 0$ such that 
$$\beta_\eta(t) \coloneqq \sup_{s \geq 0} \sup_{C \in \cF^0_s \otimes \overbar{\cF}{}^0_{t+s}} \big\lvert \PP^\eta\vert_{\cF^0_s \otimes \overbar{\cF}{}^0_{t+s}} (C) - \PP^\eta\vert_{\cF^0_s} \otimes \PP^\eta\vert_{\overbar{\cF}{}^0_{t+s}}(C) \big\rvert \leq c_\kappa \mathrm{e}^{-\kappa t}, \quad t \geq 0.$$
Here, $\PP^\eta\vert_{\cF_s^0 \otimes \overbar{\cF}{}^0_{t+s}}$ denotes the probability measure on $(\Omega \times \Omega,\cF_s^0 \otimes \overbar{\cF}{}^0_{t+s})$, defined as the image measure of $\PP^{\eta}$ under the canonical injection $\iota(\omega) = (\omega,\omega)$.
This implies that, for $A \times B \in \cF_s^0 \otimes \overbar{\cF}{}^0_{t+s}$, we have $\PP^\eta\vert_{\cF_s^0 \otimes \overbar{\cF}{}^0_{t+s}}(A \times B) = \PP^\eta(A \cap B)$. 
This demonstrates that $\beta$-mixing of a Markov process describes a form of asymptotic independence of its future and its past. 
In case $\eta = \mu$, we just write $\beta_\mu = \beta$, and it can be shown that the $\beta$-mixing coefficient reduces in this case to
$$\beta(t) = \int_{\cX} \lVert P_t(x, \cdot) - \mu \rVert_{\mathrm{TV}} \, \mu(\diff{x}),$$
$\lVert \cdot \rVert_{\mathrm{TV}}$ denoting the total variation norm. 
The central assumption in \cite{dexheimer20} is that the stationary Markov process $\X$ is exponentially $\beta$-mixing, i.e., there exist constants $\kappa,c_\kappa  > 0$ such that 
\begin{equation}\label{eq:mix}
\beta(t) \leq c_\kappa \mathrm{e}^{-\kappa t}, \quad t \geq 0.
\end{equation}
Under this condition, the following result is proved.

\begin{proposition}[Theorem 3.2 in \cite{dexheimer20}]\label{prop:mix}
Let $\mathcal{G}$ be a countable class of bounded real-valued functions $g$ satisfying $\mu(g) = 0$, and define
$$\G_T(g) = \frac{1}{\sqrt{T}}\int_0^T g(X_t) \diff{t}, \quad T > 0, g\in\mathcal G.$$
Suppose that $\X$ is stationary with invariant  distribution $\mu$ and exponentially $\beta$-mixing, and let $m_t \in (0, t\slash 4]$.
Then, there exist $\tau \in [m_t, 2m_t]$ and constants $C_1,C_2,c_1,c_2>0$ such that, for any $1\le p<\infty$,
\begin{equation} \label{eq: unimom}
\begin{split}
\left(\E^\mu\left[\sup_{g\in\mathcal G}|\G_t(g)|^p\right]\right)^{1/p}&\leq C_1\int_0^\infty\log\mathcal N\big(u,\mathcal G,\tfrac{2m_t}{\sqrt{t}}d_\infty\big)\d u+C_2\int_0^\infty\sqrt{\log \mathcal N(u,\mathcal G,d_{\G,\tau})}\diff u\\
&\quad+4\sup_{g\in\mathcal G}\Big(\frac{2m_t}{\sqrt{t}}\|g\|_\infty c_1p+ \lVert g\rVert_{\mathbb{G},\tau}c_2\sqrt p +\frac{1}{2}\lVert g \rVert_{\infty} c_\kappa \sqrt{t} \mathrm{e}^{-\frac{\kappa m_t}{p}}\Big).
\end{split}
\end{equation}
Here, for $f,g \in \mathcal{G}$, 
$$d^2_{\mathbb{G},T}(f,g) \coloneqq \mathrm{Var}\Big(\frac{1}{\sqrt{T}} \int_0^T (f-g)(X_t) \diff{t} \Big),\quad T>0,$$
defines a semi-distance and, for any semi-distance $d$ and $\varepsilon > 0$, $\mathcal{N}(\varepsilon, \mathcal{G}, d)$ denotes the covering number of $\mathcal{G}$ by $d$-balls of radius $\varepsilon$. 
\end{proposition}

This result covers a wide range of potential applications. For example, it can be used to find optimal upper bounds (regarding the $\sup$-norm risk over bounded domains) for nonparametric estimation of the invariant density for $\R^d$-valued Markov processes with transition densities (cf.~Sections 4 and 5 in \cite{dexheimer20}). 
Such results are derived from Proposition \ref{prop:mix} by bounding the (pseudo-) norms and thus the associated entropy integrals for the function class $\mathcal G$ related to the chosen estimation procedure. For $d_\infty$, this can be achieved by using the analytical properties of $\mathcal G$, while bounds on the pseudo-metric $d_{\G,\tau}$ are based on suitable bounds for the variance of integral functionals of $\X$.
As shown in \cite{dexheimer20}, for $d \geq 2$ this is taken care of using the exponential $\beta$-mixing property of $\X$ once we assume that, additionally, an on-diagonal heat kernel estimate is in place for the densities $(p_t)_{t \geq 0}$ of the Markov semigroup, i.e., there exists some constant $C > 0$ such that 
$$\forall t \in (0,1]\colon \quad \sup_{x,y\in \R^d} p_t(x,y) \leq Ct^{-d\slash 2}.$$
\paragraph{Application in the scalar setting}
It will be demonstrated in the sequel that the one-dimensional case can also be treated optimally within the framework described in Proposition \ref{prop:mix}, thereby distinguishing between two different situations.
In Section \ref{subsec: diff station}, we study kernel invariant density estimation (for scalar ergodic diffusions) which requires a careful balancing of \emph{bias} and stochastic error of the estimator (in case of pointwise risk, the well-known bias--variance tradeoff) by choosing an appropriate bandwidth $h$. 
In dimension $d=1$, the $\beta$-mixing property is not quite sufficient to guarantee variance bounds that are tight enough for proving optimal upper bounds on the convergence rates. 
We additionally require convergence of the semigroup densities to the invariant density at sufficient speed on compact sets, combined with a relaxation of the on-diagonal heat kernel estimate of the semigroup. 
In Appendix \ref{app:diff}, we demonstrate that classical assumptions on the diffusion process---as they are introduced in Section \ref{sec:diff} to ensure existence of a stationary solution---are tight enough to match the above requirements. 
This is of considerable independent interest since, in contrast to the local time arguments usually employed for the statistical analysis of scalar diffusions, the techniques generalize without much effort to the multivariate diffusion case. 
Our framework therefore arguably closes the gap between the relatively distinct approaches to statistical estimation of scalar and multivariate diffusions (see, e.g., \cite{dalrei06} vs.\ \cite{dalrei07} or \cite{dal03,dal05} vs.\ \cite{str15,str16}). 
Moreover, it potentially extends results obtained exclusively for symmetric diffusions to the general case since it is not reliant on functional inequalities, which are not well-suited to the non-reversible setup, see also the discussion in \cite{dexheimer20}.

Suppose that we can find an \emph{unbiased} estimator of the characteristic we are interested in.
In this situation, a fine analysis of the variance of Markovian functionals is not necessarily needed. 
This is, e.g., the case if we can express the quantity of interest as an integral wrt the stationary distribution of some stationary  Markov process $\X$, since then the continuous-time mean estimator 
$$\frac{1}{T} \int_0^T f(X_t) \diff{t}$$ is unbiased. If $\X$ is moreover $\beta$-mixing, we can make use of Proposition \ref{prop:mix} based on purely analytical arguments. 
This will become clear in Section \ref{subsec: ascending}, where we are investigating $\sup$-norm estimation of generator functionals $\mathcal{A}_H \gamma$ of the ascending ladder height process $\mathbf{H}$ belonging to a L\'evy process $\X$ via an unbiased mean estimator based on overshoots of $\X$. 
The thereby established $\sup$-norm bounds will be of central importance for the procedure in Section \ref{sec: data levy}. 

\subsection{Organization of the paper}
In Section \ref{sec:diff}, we develop a data-driven strategy for a singular control problem associated to a scalar diffusion process. The construction and error analysis is given in Section \ref{subsec:control_diffusion}, based on a minimax optimal estimation procedure for the stationary density under exponential $\beta$-mixing assumptions, which is carried out in Section \ref{subsec: diff station}. Appendix \ref{app:diff} complements the study of Section \ref{subsec: diff station} by demonstrating how the variance analysis of the estimator can be carried out in a self-contained fashion within the $\beta$-mixing framework without having to resort to local time arguments which do not generalize to higher dimensions. In Section \ref{sec:levy}, a data driven strategy for an impulse control problem with an underlying Lévy process is constructed. The statistical foundations for the estimation strategy of Section \ref{sec: data levy} are presented in Section \ref{subsec: ascending}. Central aspects of overshoot convergence as main ingredient to the statistical analysis are summarized in Appendix \ref{sec: appendix levy}.

\section{Data-driven singular controls for diffusions on the real line}\label{sec:diff}
We now introduce the singular control problem for underlying scalar diffusion processes, given as a solution of the It\={o}-type SDE
\begin{equation}\label{eq:dynamcs}
\d X_t\ =\ b(X_t)\d t+\sigma(X_t)\d W_t,
\end{equation}
$b,\sigma\colon \R\to\R$ some measurable functions and $(W_t)_{t\ge0}$ some standard Brownian motion on some probability space $(\Omega,\cF,\PP_b).$
 One motivation for considering such problems comes from investigating optimal dividend distributions \cite{alvarez2006class,asmussen1997controlled,cadenillas2007optimal}.
Another stream of literature deals with determining a policy that optimizes the expected cumulative present value of the harvesting \cite{alvarez1998optimal,hening2019asymptotic,lande1994optimal}. 
In particular for the latter application, it is natural to study an ergodic formulation, as it reflects the idea of considering sustainable harvesting guidelines, which we will also use here. 

Assume that for some constants $\overbar{\nu},\underbars\nu\in(0,\infty),$  $\sigma$ is continuous, differentiable and globally Lipschitz and satisfies $\underbars\nu\leq \left|\sigma(x)\right|\leq \overbar\nu$ for all $x \in \R$. For fixed constants $A,\gamma>0$ and $\C\geq 1$, define the set $\bm{\Sigma}=\bm{\Sigma}(\C,A,\gamma,\sigma)$ as
\[\bm{\Sigma}\coloneqq\Big\{b \in \operatorname{Lip}(\R):|b(x)| \leq\C(1+|x|),\ \forall|x|>A\colon \frac{b(x)}{\sigma^2(x)}\operatorname{sgn}(x)\leq -\gamma\Big\}.\]
Note that a linear growth condition for Lipschitz drift $b$ is always satisfied, but the class $\bm{\Sigma}$ specifies a global magnitude of this maximal growth in terms of the constant $\C$. Moreover, given $\sigma$ as above and any $b\in\bm{\Sigma}$, an immediate consequence is that there exists a strong solution $\X$ of the SDE \eqref{eq:dynamcs} for given initial value $X_0$ independent of $\bm{W}$. If we let $\PP^x_b = \PP_b(\cdot \vert X_0 = x)$, then $(\X,(\PP^x_b)_{x \in \R})$  defines a non-explosive Feller Markov process and thus in particular a Borel right process. Moreover, $\X$ has a unique stationary distribution $\mu=\mu_b$ having invariant density
\begin{equation} \label{eq: inv dens diff}
\rho(x)=\rho_b(x)\coloneqq \frac{1}{C_{b,\sigma}\sigma^2(x)}\exp\left(\int_0^x\frac{2b(y)}{\sigma^2(y)}\d y\right),\quad x\in \R,
\end{equation}
with normalizing constant 
$C_{b,\sigma}\coloneqq \int_{\R}\frac{1}{\sigma^2(u)}\ \exp\left(\int_0^u\frac{2b(y)}{\sigma^2(y)}\d y\right)\d u$. 
In the following, we will abbreviate $\PP^{\mu_b}_b = \PP_b$ and, if there is no room for confusion, also just write $\PP$ instead. For any $b\in\bm{\Sigma}$, $\sigma^2\rho_b$ is continuously differentiable and there exists a constant $\rho^\ast>0$ (depending only on $\C,A,\gamma,\underbars\nu,\overbar\nu$) such that 
\begin{equation}\label{eq:rho_uniformly_bounded1}
\sup_{b\in\bm{\Sigma}(\C,A,\gamma,\sigma)}\max\left\{\|\rho_b\|_{\infty},\ \|(\sigma^2\rho_b)'\|_\infty\right\} < \rho^\ast.
\end{equation}
{
Furthermore, for any fixed bounded set $D\subset \R$, there exists some $\rho_\ast>0$ (depending again only on $\bm{\Sigma}$) such that
\begin{equation}\label{eq:rho_uniformly_bounded2}
\forall x\in D,\quad \inf_{b\in\bm{\Sigma}(\C,A,\gamma,\sigma)}\rho_b(x)\ge \rho_\ast.
\end{equation}
}

The controls used to formulate the problem are of the form $\mathbf Z=(U_t,D_t)_{t\geq0}$ for non-decreasing, right-continuous and adapted processes $\mathbf U$ and $\mathbf D$.
Here, $U_t$ and $D_t$ denote the cumulative upwards and downwards controls,  resp. These processes can be decomposed into singular and jump part as 
\[U_t=U_t^c+\sum_{0\leq s \leq t} (U_s-U_{s-}),\;D_t=D_t^c+\sum_{0\leq s \leq t} (D_s-D_{s-}),\]
where $\mathbf U^c$ and $\mathbf D^c$ are continuous. In the following, we will mostly deal with a special class of  controls for which the jump part is absent (with a possible exception at $t=0$): $\mathbf U$ and $\mathbf D$ are associated to the local times at certain fixed points $c,d$.


We denote the set of all controls by $\bm \Lambda$ and, for each $\mathbf Z\in \bm\Lambda$, we define the controlled process $\mathbf X^Z$ as the solution to 
\begin{equation*}
	\d X_t^Z\ =\ b(X_t^Z)\d t+\sigma(X_t^Z)\d W_t+\diff U_t-\diff D_t,
\end{equation*}
where we work under the assumption that $b\in\bm{\Sigma}$, implying in particular that the uncontrolled process $\mathbf X= \mathbf X^0$ has a stationary distribution $\rho=\rho_b$. 

The problem to be studied is now to determine the minimal value and the minimizer of
\begin{align}\label{eq:problem}
	\limsup_{T\to\infty}\frac{1}{T}\left(\int_0^Tc(X^Z_s)\d s + q_uU_T+q_dD_T\right),
\end{align}
where $c$ is a continuous, nonnegative function with $$\sup_{b\in\bm{\Sigma}}\int c(x) \rho_b(x)\d x<\infty$$
modelling the running costs and $q_u,q_d$ are positive constants describing the (proportional) costs associated with applying a control. 
We can interpret our goal as keeping $\X$ close to the target state $0$, say, and therefore assume that $c$ has a minimum in 0. 
The goal in the sequel is to find a data-driven strategy for problem \eqref{eq:problem} when the drift $b$ of the underlying process is unknown. 
While parts of the following analysis are similar to the one in \cite{christensen20}, it here turns out to be essential to control the $\sup$-norm risk of estimators of the characteristics (precisely, the invariant density $\rho_b$) of $\X$ solving \eqref{eq:dynamcs}.

\subsection{Estimating the stationary density of ergodic diffusion processes} \label{subsec: diff station}
We first show how the underlying statistical problem for the uncontrolled process can be integrated into our general framework presented in Section \ref{sec:2}.
For a large class of ergodic scalar diffusion processes $\X$ solving \eqref{eq:dynamcs} it is known that, given continuous observations $(X_t)_{0\le t\le T}$, the invariant density can be estimated with a parametric rate of convergence. For an overview, we refer to Sections 1.3.2 and 4.2 in \cite{kut04}.
It is however not straightforward to extend bounds on the pointwise or $L^2$ risk to the $\sup$-norm bounds required for our application.
A corresponding result is given in Corollary 13 in \cite{aeckerle18} whose proof, however, relies substantially on the use of diffusion local time.
We show how this behaviour can also be deduced from mixing properties of the diffusion.

Given some fixed domain $D\subseteq \R$, constants $\beta,\coa,\cob,\coc>0$ and a measurable function $V \geq 1$, introduce the set $\tilde{\bm{\Sigma}}_D(\beta)=\tilde{\bm{\Sigma}}_D(\beta,\coa,\cob,\coc)$
\begin{equation}\label{def:tildesigma}
\tilde{\bm{\Sigma}}_D(\beta)\coloneqq 
\bigg\{b\in\bm{\Sigma}: \ \|P_t(x,\cdot)-\mu_b\|_{\operatorname{TV}}\le \cob  V(x)\e^{-\coa t} \text{ with } \mu_b(V) \le\coc\text{ and }\rho_b\in C^\beta(D)\bigg\}.
\end{equation}
Note that for diffusions $\X$ with drift $b \in \tilde{\bm{\Sigma}}_D(\beta)$ it holds that $\X$ is exponentially ergodic, i.e., the total variation distance between the marginal laws of $\X$ and the invariant distribution decreases exponentially fast in time, and  $\X$ is exponentially $\beta$-mixing with mixing coefficient $\beta(t) \leq \cob \coc \mathrm{e}^{-\coa t}$, which is independent of $b$. As demonstrated in Proposition \ref{lem:rfunction} and Lemma \ref{lem: diff mix}, exponential ergodicity and the exponential $\beta$-mixing property are satisfied for any $\X$ such that the coefficients $b \in \bm{\Sigma}$ and $\sigma$ are globally Lipschitz. Apart from the assumption on the H\"older continuity of $\rho_b$, restricting the class $\bm{\Sigma}$ to $\tilde{\bm{\Sigma}}_D$ should therefore be understood as a technical device to obtain \emph{uniform} control on the coefficients in the exponential $\beta$-mixing bound, which is needed for the upper bound in the minimax sense provided in Theorem \ref{theo:invdens}. In the following, we consider the supremum norm $\left\|f\right\|_{\infty}=\left\|f\right\|_{D,\infty}$ on $D$. 

\begin{theorem}[concentration of invariant density estimators]\label{theo:invdens}
Fix some open and bounded set $D\subset \R$, assume that $b\in\tilde{\bm{\Sigma}}_D(\beta+1)$, for some $\beta>0$, and let $Q$ be a compactly supported, Lipschitz-continuous kernel function of order $\lfloor\beta+1\rfloor$. 
Define the estimator 
\begin{equation}\label{def:rhohat}
\hat\rho_T(x)\coloneqq \frac{1}{\sqrt{T}(\log T)^2}\int_0^TQ\left(\frac{\sqrt T(x-X_u)}{(\log T)^2}\right)\d u,\quad x\in D.
\end{equation} 
Then, for any $p\geq 1$,
\begin{equation}\label{eq: unif rate station}
	\sup_{b\in\tilde{\bm{\Sigma}}_D(\beta+1)}\left(\E^{\mu_b}_b\left[\left\|\hat\rho_T - \rho_b\right\|^p_\infty\right]\right)^{1/p}
\in \mathsf{O}\Big(\sqrt{\tfrac{\log T}{T}}\Big).
\end{equation}
\end{theorem}

\begin{proof}
Fix $p\in[1,\infty)$ and $b \in \tilde{\bm{\Sigma}}(\beta +1)$, and denote $h=h_T\coloneqq (\log T)^2/\sqrt T$, $m_T\coloneqq p\log T/\coa $, $\E_b = \E$, $\rho_b = \rho$
\begin{align*}
\mathcal G&\coloneqq \left\{Q\left(\frac{x-\cdot}{h}\right)-\E\Big[ Q\Big(\frac{x-X_0}{h}\Big)\Big]: x\in D\cap \mathbb Q\right\},\\
\mathbb H_T(x)&\coloneqq\hat\rho_T(x)-\E\big[\hat\rho_T(x)\big]=\frac{1}{\sqrt Th}\mathbb G_T\left(Q\Big(\frac{x-\cdot}{h}\Big)-\E\Big[ Q\Big(\frac{x-X_0}{h}\Big)\Big]\right).
\end{align*}
Given any $b\in\tilde{\bm{\Sigma}}_D(\beta+1)$, it follows from Proposition \ref{lem:rfunction} and Lemma \ref{lem:prop21} that the associated diffusion process solving \eqref{eq:dynamcs} is exponentially $\beta$-mixing. 
Thus, we may apply Proposition \ref{prop:mix} for bounding
\begin{equation}\label{def:mathbb H}
\left(\E\left[\sup_{x\in D}|\mathbb H_T(x)|^p\right]\right)^{1/p}=\left(\E\left[\sup_{x\in D\cap\mathbb Q}|\mathbb H_T(x)|^p\right]\right)^{1/p}.
\end{equation}
Let $\tau$ as in Proposition \ref{prop:mix}, and denote by $L_T(y)$ the local time of $\X$ at the point $y\in\R$ up to time $T\ge0$, fulfilling in particular $\E[L_T(y)]=T\rho(y)\sigma^2(y)$.
Using the occupation times formula and Minkowski's integral inequality, one obtains
\begin{equation}
\begin{split}\label{eq:Qbound}
\Var\left(\int_0^\tau Q\left(\frac{x-X_u}{h}\right)\diff{u}\right)&
= \E\left[\left(\int_{\R} Q\left(\frac{x-y}{h}\right)\left(\frac{L_\tau(y)}{\sigma^2(y)}-\tau\rho(y)\right)\diff{y}\right)^2\right]\\
&= h^2\ \E\left[\left(\int_{\R} Q(v)\left(\frac{L_\tau(x-hv)}{\sigma^2(x-hv)}-\tau\rho(x-hv)\right)\diff{v}\right)^2\right]\\
&\le h^2\int_{\R} Q(v)\ \E\left[\left(\frac{L_\tau(x-hv)}{\sigma^2(x-hv)}-\tau\rho(x-hv)\right)^2\right]\diff{v}\\
&\le h^2\sup_{v\in\operatorname{supp}(Q)}\frac{\Var\left(L_\tau(x-hv)\right)}{\sigma^4(x-hv)}\\
&\le\ h^2 \underbars\nu^{-4}C_0\tau\sup_{v\in\operatorname{supp}(Q)}\rho(x-hv),
\end{split}
\end{equation}
where the last estimate follows from Proposition 5.1 in \cite{dalrei06} and $C_0>0$ is a constant depending only on the class $\bm{\Sigma}$.
Thus, 
\[\sup_{f,g\in\mathcal G}d_{\G,\tau}(f,g)=\sup_{f,g\in\mathcal G}\sqrt{\Var\left(\frac{1}{\sqrt \tau}\int_0^\tau(f-g)(X_s)\diff{s}\right)}\le h\mathfrak D, \quad\text{ for }\mathfrak{D}\coloneqq \underbars\nu^{-2}\sqrt{C_0\rho^\ast},
\]
such that $\mathcal N(\ep,\mathcal G,d_{\G,\tau})=1$ for $\ep\ge h\mathfrak D$.
Similarly, for any $g\in\mathcal G$,
\[\Var\left(\frac{1}{\sqrt \tau}\int_0^\tau g(X_u)\diff{u}\right)\le \sup_{x\in\operatorname{supp}(g)}\frac{\Var(L_\tau(x))}{\tau\sigma^4(x)\rho(x)}\ \|g\|_{L^2(\mu)}^2\le
\mathfrak D^2\rho_\ast^{-1}\|g\|_{L^2(\mu)}^2.
\]
Assuming that $h\le\mathfrak{A}/(8\sqrt{\rho_\ast})$ and denoting by $L$ the Lipschitz constant of $Q$, it then follows from Lemma \ref{lemma: covering numbers} that
\begin{align*}
\int_0^\infty\sqrt{\log\mathcal N(\ep,\mathcal G,d_{\G,\tau})} \d \ep&\le \int_0^{h\mathfrak D}\sqrt{\upsilon\log\left(\frac{\mathfrak{A D}}{\sqrt{\rho_\ast}\ep}\right)}\diff{\ep}\ \le\
4\mathfrak D h\sqrt{\upsilon\log\left(\frac{\mathfrak{A}}{h\sqrt{\rho_\ast}}\right)},\\
\int_0^\infty\log \mathcal N\big(\ep,\mathcal G,\tfrac{2m_T}{\sqrt{T}}d_\infty\big)\d\ep&\leq \frac{8m_T}{\sqrt{T}}\lVert Q \rVert_\infty \Big(1 + \log \Big(\frac{L \mathrm{diam}(D)}{\Vert Q \rVert_\infty h} \Big) \Big).
\end{align*}
Thus, Proposition \ref{prop:mix} gives that \eqref{def:mathbb H} is upper-bounded by
\begin{equation}\begin{split}\label{eq:stocherror}
&\frac{1}{\sqrt Th}\bigg(8 C_1 \frac{m_T}{\sqrt{T}}\lVert Q \rVert_\infty \Big(1 + \log \Big(\frac{L \mathrm{diam}(D)}{\Vert Q \rVert_\infty h} \Big) \Big)
+4 C_2 \mathfrak D h\sqrt{\upsilon\log\left(\frac{\mathfrak{A}}{h\sqrt{\rho_\ast}}\right)}\\&\hspace*{7em}+\frac{16m_T}{\sqrt{T}}\|Q\|_\infty c_1 p+2h\mathfrak D c_2\sqrt p + 4\lVert Q\rVert_\infty \cob\coc\sqrt{T} \mathrm{e}^{-\frac{\coa m_T}{p}}\bigg) \in\mathsf{O}\Big(\sqrt{\tfrac{\log T}{T}}\Big),
\end{split}\end{equation}
where the last implication follows from the choice of $h=h_T$ and $m_T$.
The conditions on the order of the kernel function $Q$ and the fact that $\beta>0$ further imply that, for any $x\in D$,
\[
\left|\E[\hat\rho_T(x)]-\rho(x)\right|=\left|(\rho\ast Q_h-\rho)(x)\right| \lesssim h^{\beta+1} \in\mathsf{O}\Big(\sqrt{\tfrac{\log T}{T}}\Big).\]
In combination with the upper bound on the stochastic error stated in \eqref{eq:stocherror}, we thus obtain \eqref{eq: unif rate station}.
\end{proof}
For the sake of a concise presentation, we have here restricted to demonstrating how a bound for the $\sup$-norm risk can be derived directly from Proposition \ref{prop:mix}. 
Main components for doing so are the verification of the mixing property and the derivation of bounds of the variance of integral functionals. 
Further details on these two steps are summarised in Section \ref{app:diff}. 
In particular, we show there how a variance bound as in \eqref{eq:Qbound} can be derived \emph{without} relying on specific properties of diffusion local time. 

Although assuming stationarity of the process is standard in the statistical literature, this assumption can be slightly problematic for practical purposes. In the present scenario this will become evident for our proof technique of the data-driven control strategy, where we require $\sup$-norm bounds under $\PP^0_b$ instead of $\PP^{\mu_b}_b$. To extend the rate result \eqref{eq: unif rate station} from the stationary regime to the non-stationary case, we use the following auxiliary result, which shows that exponential convergence allows to exactly quantify the loss imposed by non-stationarity for nonparametric estimation. We focus on the case $p=1$ as the relevant result for our purposes.

\begin{lemma}\label{lem:conv int}
Let $\mathcal{X}$ be a topological space and $(\mathbf{Y},(\PP^x)_{x \in \mathcal{X}})$ an $\mathcal{X}$-valued exponentially ergodic Markov process, i.e., there exist a function $V\colon \mathcal{X} \to [1,\infty)$ and constants $c,\kappa > 0$ such that, for any $x \in \mathcal{X}$,
$$\big\Vert \PP^x(Y_t \in \cdot) - \mu \big\Vert_{\mathrm{TV}} \leq cV(x)\mathrm{e}^{-\kappa t}, \quad t \geq 0,$$
where $\mu$ is the invariant distribution of $\mathbf{Y}$. Then, for any bounded $g \in \mathcal{B}(\mathcal{X}^2)$, $x \in \mathcal{X}$ and $T$ large enough such that $T > \kappa^{-1} \log T$, it holds that 
\begin{align*}
&\Big\vert \E^x\Big[\sup_{y \in \mathcal{X}}\Big\vert \frac{1}{T} \int_0^T g(y,Y_s) \diff{s}\Big\vert \Big] - \E^\mu\Big[\sup_{y \in \mathcal{X}}\Big\vert \frac{1}{T} \int_0^T g(y,Y_s) \diff{s}\Big\vert \Big] \Big\vert\\
&\quad \leq \lVert g \rVert_\infty \Big(\frac{2 \log T}{\kappa T} + cV(x) \frac{1}{T} \Big).
\end{align*}
\end{lemma}
\begin{proof}
Let 
$$\theta(y,u,v) = \int_{u}^v g(y,X_s) \diff{s}, \quad 0 \leq u \leq v, y \in \mathcal{X},$$
and 
$$\varphi(x) \coloneqq \E^x\big[\lVert\theta(\cdot,0,T-\kappa^{-1}\log T)\rVert_\infty\big], \quad x \in \mathcal{X}.$$
Then, 
\begin{align*} 
&\Big\vert \E^x\Big[\sup_{y \in \mathcal{X}}\Big\vert \frac{1}{T} \int_0^T g(y,Y_s) \diff{s}\Big\vert \Big] - \E^\mu\Big[\sup_{y \in \mathcal{X}}\Big\vert \frac{1}{T} \int_0^T g(y,Y_s) \diff{s}\Big\vert \Big] \Big\vert\\
&\leq \frac{1}{T} \Big( \E^x\big[\big\lvert \lVert \theta(\cdot, 0, T)\rVert_\infty - \lVert \theta(\cdot, \kappa^{-1}\log T, T)\rVert_\infty \big\vert\big] + \E^\mu\big[\big\lvert \lVert \theta(\cdot, 0, T)\rVert_\infty - \lVert \theta(\cdot, \kappa^{-1}\log T, T)\rVert_\infty\big\vert\big]\Big)\\
&\quad+ \frac{1}{T} \Big\vert\E^x\big[\lVert\theta(\cdot,\kappa^{-1}\log T, T)\rVert_\infty\big] - \E^\mu\big[\lVert\theta(\cdot,\kappa^{-1}\log T, T)\rVert_\infty\big] \Big\vert\\
&\leq \frac{1}{T} \Big( \E^x\big[\lVert \theta(\cdot, 0, \kappa^{-1} \log T)\rVert_\infty \big] + \E^\mu\big[\lVert \theta(\cdot, 0, \kappa^{-1} \log T)\rVert_\infty \big]\Big)\\
&\quad+ \frac{1}{T} \Big\vert\E^x\big[\lVert\theta(\cdot,\kappa^{-1}\log T, T)\rVert_\infty\big] - \E^\mu\big[\lVert\theta(\cdot,\kappa^{-1}\log T, T)\rVert_\infty\big] \Big\vert\\
&\leq 2\lVert g \rVert_\infty \frac{\log T}{\kappa T} + \frac{1}{T} \Big\vert\E^x\big[\lVert\theta(\cdot,\kappa^{-1}\log T, T)\rVert_\infty\big] - \E^\mu\big[\lVert\theta(\cdot,\kappa^{-1}\log T, T)\rVert_\infty\big] \Big\vert\\
&= 2\lVert g \rVert_\infty \frac{\log T}{\kappa T} + \frac{1}{T}\big\vert\E^x\big[\varphi(Y_{\kappa^{-1}\log T})\big] - \mu(\varphi)\big\vert,
\end{align*}
where for the second inequality we used reverse triangle inequality for the first two summands and the last equality is a consequence of the Markov property of $\mathbf{Y}$ and stationarity of $\mathbf{Y}$ under $\PP^\mu$. Using that $\lVert \varphi \rVert_\infty \leq \lVert g \rVert_\infty T$, exponential ergodicity of $\mathbf{Y}$ yields 
\begin{align*} 
&\Big\vert \E^x\Big[\sup_{y \in \mathcal{X}}\Big\vert \frac{1}{T} \int_0^T g(y,Y_s) \diff{s}\Big\vert \Big] - \E^\mu\Big[\sup_{y \in \mathcal{X}}\Big\vert \frac{1}{T} \int_0^T g(y,Y_s) \diff{s}\Big\vert \Big] \Big\vert\\
&\leq 2\lVert g \rVert_\infty\frac{\log T}{\kappa T} + cV(x)\lVert g \rVert_\infty \mathrm{e}^{-\kappa  (\kappa^{-1} \log T)}\\
&= 2\lVert g \rVert_\infty\frac{\log T}{\kappa T} + cV(x)\lVert g \rVert_\infty \frac{1}{T},
\end{align*}
as claimed.
\end{proof}
With this at hand, we obtain a non-stationary version of Theorem \ref{theo:invdens}.

\begin{corollary}\label{coro:invdens}
Given the assumptions from Theorem \ref{theo:invdens}, it holds for any $x \in \R$ that 
\begin{equation}\label{eq: unif rate nonstation}
\sup_{b\in\tilde{\bm{\Sigma}}_D(\beta+1)}\E^x_b\left[\left\|\hat\rho_T - \rho_b\right\|_\infty\right]
\in \mathsf{O}\Big(\sqrt{\tfrac{\log T}{T}}\Big).
\end{equation}
\end{corollary}
\begin{proof}
Let 
$$g(x,y) \coloneqq \frac{\sqrt{T}}{(\log T)^2}Q\Big(\frac{\sqrt{T}(x-y)}{(\log T)^2} \Big) - \rho_b(x), \quad x,y \in \R.$$
Then, for  $T$ large enough such that $\sqrt{T}/(\log T)^2\lVert Q\rVert_\infty \geq \rho^\ast$, we have
$$\lVert g \rVert_\infty \leq 2 \sqrt{T}/(\log T)^2\lVert Q\rVert_\infty.$$
Applying Lemma \ref{lem:conv int} to $g$ and $\mathbf{X}$, which is exponentially ergodic by construction of $\tilde{\bm{\Sigma}}_D(\beta +1)$, we obtain, for any $x \in \R$ and $T$ large enough such that $T > \bm{\vartheta}^{-1}\log T$, that 
\begin{align*}
\left\vert \E^x_b\left[\left\|\hat\rho_T - \rho_b\right\|_{\infty}\right] - \E^{\mu_b}_b\left[\left\|\hat\rho_T - \rho_b\right\|_{\infty}\right] \right\vert \leq 2\lVert Q \rVert_\infty \Big(\frac{2}{\bm{\vartheta} \sqrt{T} \log T} + \frac{\bm{\varrho} V(x) }{\sqrt{T} (\log T)^2} \Big).
\end{align*}
Combining this with \eqref{eq: unif rate station}, we obtain \eqref{eq: unif rate nonstation} by triangle inequality.
\end{proof}

\subsection{Application}\label{subsec:control_diffusion}
We now turn to analyzing the singular control problem. 
Given the literature on related problems, it is natural to expect reflecting barrier strategies, i.e., strategies which maintain the process between two constant thresholds $\colow$ and $\coup$ and just grow at these points, to be optimal. We denote such strategies for upper and lower, resp., boundaries $\colow,\coup\in\R,\,\colow<\coup,$ by $\mathbf U^{\colow}$ and $\mathbf D^{\coup}$ and refer to \cite[\S 23]{gisk72} for an explicit characterization, which underlines the interpretation of these processes as local times at the boundaries $\xi$ and $\theta$, resp., of the process $\X^Z$.

\subsubsection{Solution for known drift}\label{sec:known_b}
Using classical ergodic results for one-dimensional linear diffusions, it is straightforward to show the following analytic expression for the expected costs when applying reflecting barrier strategies.
We refer to \cite[Proposition 2.1]{alvarez2018class} for a detailed proof.
\begin{lemma}
Let $\colow,\coup\in\R,\;\colow<\coup$ and $x\in [\colow,\coup]$. Then, for $\mathbf Z=(\mathbf U^\colow, \mathbf D^\coup)$,
\[\lim_{T\to\infty}\frac{1}{T}\E^x\left[\int_0^Tc(X^Z_s)\d s + q_uU_T^\colow+q_dD_T^\coup\right]=C(\colow,\coup),\]
with
\[C(\colow,\coup)=\frac{1}{\int_{\colow}^{\coup}m(x)\d x}\left(\int_{\colow}^{\coup}c(x)m(x)\d x+\frac{q_u}{S'(\colow)}+\frac{q_d}{S'(\coup)}\right),
\]
where $m$ denotes the speed density and $S$ the scale function of the underlying diffusion. 
\end{lemma}
For our later purposes, the main observation is that---given the volatility $\sigma$---the expected cost function $C$ can completely be described in terms of the invariant density $\rho$ of the underlying diffusion. Indeed:
\[C(\colow,\coup)=\frac{1}{\int_{\colow}^{\coup}\rho(x)\d x}\left(\int_{\colow}^{\coup}c(x)\rho(x)\d x+\frac{q_u\sigma^2(\colow)}{2}\rho(\colow)+\frac{q_d\sigma^2(\coup)}{2}\rho(\coup)\right).
\]
Therefore, minimizers of $C$ correspond to optimizers of \eqref{eq:problem} in the class of reflecting barrier strategies. 
The next natural question is whether such minimizers are indeed optimal within the class of all admissible strategies, i.e., whether the minimal value in \eqref{eq:problem} is equal to
\[\Val\coloneqq \min _{(\colow,\coup)}C(\colow,\coup).\]
This also holds under natural assumptions as can be proven, e.g., adapting the lines of argument in \cite{liang2020ergodic,christensen2020solution} to the two-sided case. We, however, do not go into detail here to not overburden our paper with technicalities, but restrict our attention to the class of reflecting barrier strategies in the following. 

As 0 is our target state, it is furthermore natural that 0 is contained in the no-action-region which is assumed to be bounded. 
More precisely, we assume that there exists $\Bd>0$ such that the minimizer $(\colow^*,\coup^*)$  of $C$ fulfill 
\[(\colow^*,\coup^*)\in K_{\Bd}\coloneqq \{(\colow,\coup):-\Bd\leq \colow\leq -1/\Bd,\,1/\Bd\leq \coup\leq \Bd\}.\]
In \cite{alvarez2018class}, a natural set of assumptions is introduced to guarantee that $(\colow^*,\coup^*)$ is characterized as the unique critical point of the function $C$.
We, however, do not need uniqueness for our purposes. 

\subsubsection{Construction of the estimators}\label{subsec:estimator}
We proceed by constructing estimators $\hat \colow_T$ and $\hat \coup_T$ for the optimal thresholds $\colow^*$ and $\coup^*$ which are based on the estimator $\hat\rho_T$ of the invariant density $\rho=\rho_b$ (see \eqref{def:rhohat}). To this end, we fix some $\beta>0$, set $D=K_B$, and write $\tilde{\bm{\Sigma}}\coloneqq\tilde{\bm{\Sigma}}_D(\beta+1)$.
In principle, we just use the plug-in estimator, taking however into account that (cf.~\eqref{eq:rho_uniformly_bounded2})
\[a\coloneqq \inf_{b\in\tilde{\bm{\Sigma}}}\min_{x\in K_{\Bd}}\rho_b(x)>0.\]
This leads to the estimator
\[
\estcost\coloneqq\frac{1}{\int_{\colow}^{\coup}\hat\rho_T(x)\vee a\d x}\left(\int_{\colow}^{\coup}c(x)\hat\rho_T(x)\d x+\frac{q_u\sigma^2(\colow)}{2}\hat\rho_T(\colow)+\frac{q_d\sigma^2(\coup)}{2}\hat\rho_T(\coup)\right)\]
for the expected value $C(\colow,\coup)$ of a reflection strategy with barriers $\colow,\coup$, yielding
\[\estbound \in\arg\min_{(\colow,\coup)\in K_{\Bd}}\estcost\]
as our estimator for the optimal thresholds.
Using this, we obtain that the expected costs, when using the strategy based on the estimator after having observed the uncontrolled process for $T$ time units, converge to the optimal value with rate $\sqrt{{\log T}/{T}}$.
\begin{proposition}\label{prop:nonparam_estimator}
	For any $x \in \R$, there exists $C_1>0$ such that 
	\[
		\sup_{b\in\tilde{\bm{\Sigma}}}\E^x_b\left[C\estbound-\Val \right]\leq C_1\sqrt{\frac{\log T}{T}}.
	\]
\end{proposition}
\begin{proof}
It holds that
\begin{align*}
	C\estbound-\Val &=C\estbound-\hat{C}_T\estbound+\hat{C}_T\estbound-\min_{(\colow,\coup)\in K_{\Bd}}C(\colow,\coup)\\
	&=C\estbound-\hat{C}_T\estbound+\min_{(\colow,\coup)\in K_{\Bd}}\estcost-\min_{(\colow,\coup)\in K_{\Bd}}C(\colow,\coup)\\
	&\leq 2\max_{(\colow,\coup)\in K_{\Bd}}\left|C(\colow,\coup)-\estcost\right|.
\end{align*}
To analyze this quantity, we denote numerator and denominator of $C$ and $\hat{C}_T$ by $A_\rho,B_\rho$ and $A_{\hat\rho_T},B_{\hat\rho_T}$, resp., and obtain for all $(\colow,\coup)\in K_{\Bd}$
\begin{align*}
	\left|C(\colow,\coup)-\estcost\right|&\leq \left|\frac{A_\rho(\colow,\coup)-A_{\hat\rho_T}(\colow,\coup)}{B_\rho(\colow,\coup)}\right|+\left|\frac{A_{\hat\rho_T}(\colow,\coup)}{B_\rho(\colow,\coup)}-\frac{A_{\hat\rho_T}(\colow,\coup)}{B_{\hat\rho_T}(\colow,\coup)}\right|\\
	&\leq \frac{B}{2a}\left|{A_\rho(\colow,\coup)-A_{\hat\rho_T}(\colow,\coup)}\right|+|A_{\hat\rho_T}(\colow,\coup)|\left|\frac{1}{B_\rho(\colow,\coup)}-\frac{1}{B_{\hat\rho_T}(\colow,\coup)}\right|
\end{align*}
Now, \eqref{eq:rho_uniformly_bounded1} yields that we find an absolute constant $\mathfrak M$ such that
\[\sup_{b\in\tilde{\bm{\Sigma}}}\E^x_b\left[C\estbound-\Val\right]\leq \mathfrak M\sup_{b\in\tilde{\bm{\Sigma}}}\E^0_b\left[\left\|\hat\rho_T - \rho_b\right\|_\infty\right],\]
{proving the claim by Corollary \ref{coro:invdens}.}
\end{proof}

\subsubsection{Data-driven singular controls}\label{sec:data_driven}
In most real world applications, the decision maker is faced with the problem of collecting data about the underlying dynamics and finding the optimal strategy at the same time. Here, however, a classical trade-off between exploration and exploitation occurs. On the one hand, the decision maker wants to minimize her expected costs and therefore uses singular control strategies with an optimal estimated threshold. On the other hand, using such a greedy strategy all the time, the decision maker can't learn about the drift $b$ of the underlying process outside the estimated control interval and therefore this procedure cannot even be expected to converge. 

Our solution is to separate exploration and exploitation periods as follows (see Figure \ref{fig:plot_expl_exploit}): At the beginning of every period except the first, the process is in the target state 0.
In the exploration periods, we then let the process run uncontrolled and the period ends when the process again reaches 0 after having visited two predefined boundaries $\colow_0,\coup_0,\,\colow_0<0<\coup_0$. 

In the exploitation periods, we use an estimator for $\rho$ as defined in the previous section in order to choose suitable thresholds based on the observations. The exact specification for this estimator $\estbound$ is given below. An exploitation period ends after the process has been reflected at both the upper and lower estimated boundary and has returned to 0. In the following, we will always set $\colow_0 = -B = -\coup_0$.

We combine exploration and exploitation periods using a (deterministic) sequence $(c_n)_{n \in \N} \in \{0,1\}^{\N}$, where $c_n=0$ (and $c_n=1$) means that the $n$-th period is of exploration-type (and exploitation-type, resp.) and denote the corresponding strategy by $\hat{\mathbf Z} = (\hat{\mathbf{U}},\hat{\mathbf{D}})$. By $\tau_0=0<\tau_1<\tau_2<\ldots$ we denote the stopping times separating the periods defining $\hat{\mathbf Z}$.
The question now is how to balance the time spent for exploration and exploitation. 
A suitable choice can be made by taking into account the estimation error bounds from the previous section and balancing the errors from misspecifying $\estbound$ due to the estimation error and the losses due to the lack of control in the exploration periods. As we will see below, a suitable choice are sequences $(c_n)_{n \in \N}$ such that
 there exists $\mathfrak{d} > 0$ with 
 \begin{align}\label{eq:c-condition}
	n^{2/3} \leq \# \{i \leq n: c_i = 0\} \leq n^{2/3} + \mathfrak{d}.
\end{align}
Observe that for such a sequence there exists $\overbar{M} > 0$ such that 
\begin{align}\label{eq:c-condition1}
\# \{i \leq n: c_i = 0\} \leq \overbar{M}n^{2/3}.
\end{align}

Note that $\mathbf W$ is a Brownian motion for the filtration generated by $\mathbf W$ and the independent random variable $X_0$. With respect to this filtration, the times separating the different periods are stopping times.
Therefore, the process $\mathbf{\tilde W}$ which is constructed by putting together the paths of $\mathbf W$ in the exploration periods, is again a Brownian motion. As the process $\mathbf{\tilde X}$ which is constructed by joining the paths of $\X$ in the exploration periods fulfills $\tilde X_0=X_0$ and solves the SDE
\begin{equation*}
	\d \tilde X_s\ =\ b(\tilde X_s)\d s+\sigma(\tilde X_s)\d \tilde W_s, \quad s \geq 0,
\end{equation*}
it has the same dynamics as the uncontrolled process $\X$. 

We denote the estimator for the optimal threshold from Section \ref{subsec:estimator} for the uncontrolled process $\mathbf{\tilde X}$ until time $s$ by $\estboundcontrol$ and define
$\estbound \coloneqq (\tilde{\colow}_{S_T\wedge mT^{2/3}},\tilde{\coup}_{S_T\wedge mT^{2/3}}),$
where $S_T$ denotes the time that the controlled process $\mathbf{X}^{\hat{\mathbf Z}}$ has spent in the exploration periods until $T$, and $m$ is a constant specified in the following lemma. In other words, we base the estimator $\estbound$ for the threshold used in the exploitation periods just on the observations in the exploration periods and, in addition, just for technical reasons, ignore all observations after time $s=mT^{2/3}$.

%

\begin{figure}[h]
	\centering
	\includegraphics[scale=0.5]{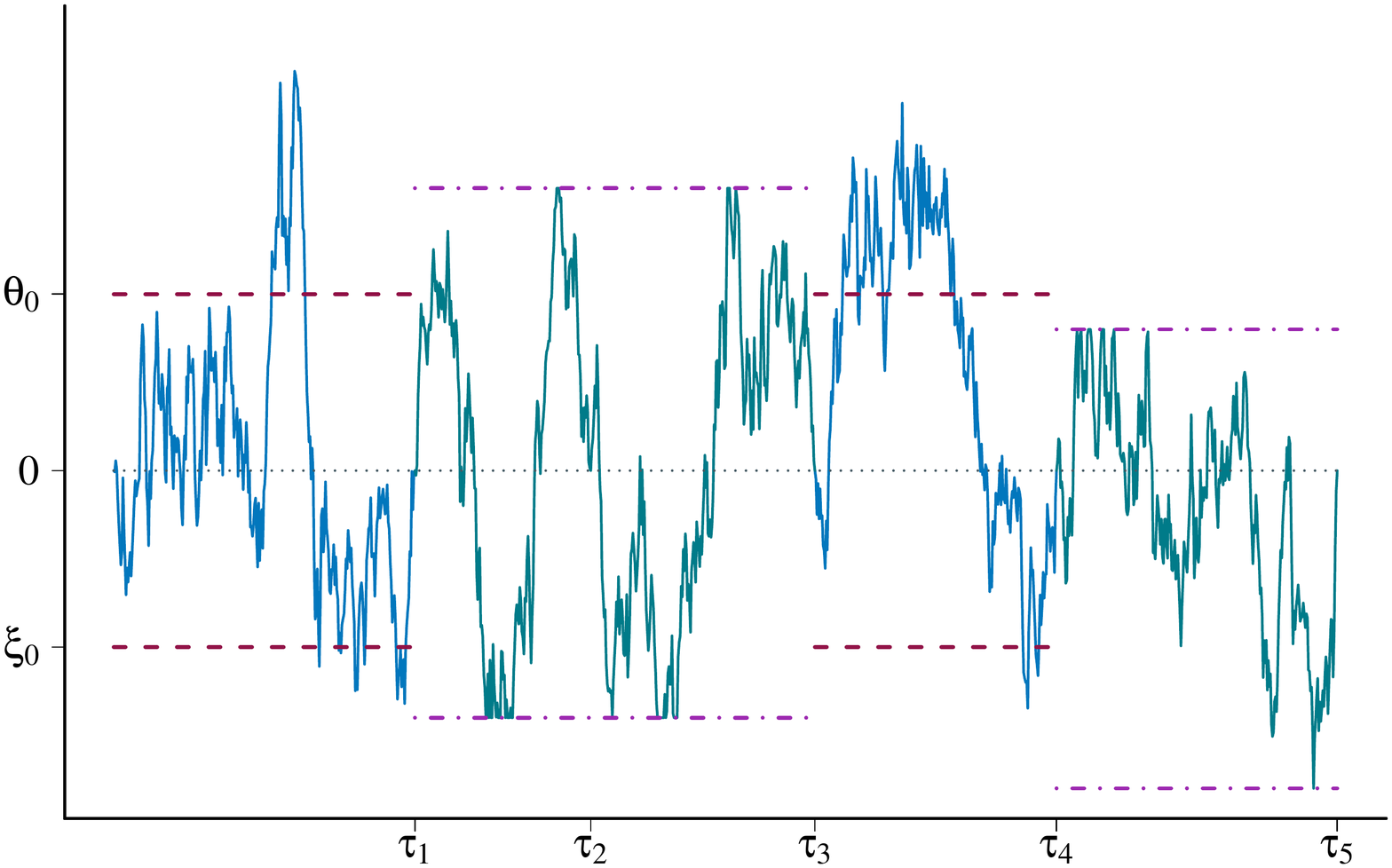}
	\caption{A path controlled using a data-driven reflections strategy with exploration (blue) and exploitation (turquoise) periods using $(c_n)_n=(0,1,1,0,1,\dots)$. The predefined boundaries $\xi_0,\theta_0$ determining the length of the exploration periods are represented by red lines and the estimated optimal reflection boundaries by purple lines.}
	\label{fig:plot_expl_exploit}
\end{figure}


We first observe that condition \eqref{eq:c-condition} implies that the time $S_T$ spent in the exploration periods until time $T$ is of order $T^{2/3}$. In particular, $S_T\to\infty$ and ${S_T}/{T}\to0$.
More precisely:
\begin{lemma}\label{lem:strategy}
Let $(c_n)_{n \in \N} \in \{0,1\}^{\N}$ satisfy \eqref{eq:c-condition}	with corresponding data-driven strategy $\hat{\mathbf Z}$ as specified above.
Then, there exist $m,M > 0$ such that 
$$\PP^0_b(T^{-2/3}S_T \leq m)\lesssim T^{-1/3} \text{ and } 
\limsup_{T \to \infty} T^{-2/3}\E^0_b[N^0_T] \leq M,$$
where $N^0_T$ denotes the number of exploration periods until time $T$.
\end{lemma}
The proof, which is quite technical and based on renewal theoretic arguments, is deferred to Appendix \ref{sec:proof strategy}. The main result of this section given below shows that, by employing the strategy $\hat{\mathbf{Z}}$, we can guarantee that the expected regret per time unit vanishes with rate $\sqrt{\log T}/T^{1/3}$.

\begin{theorem}\label{thm:main_}
Let $(c_n)_{n \in \N} \in \{0,1\}^{\N}$ satisfy \eqref{eq:c-condition} with corresponding data-driven strategy $\hat{\mathbf Z}$ as specified above.
Then, the expected regret per time unit is of order $\mathsf{O}\bigg(\frac{\sqrt{\log T}}{T^{1/3}}\bigg)$. 
That is, for any $b \in \tilde{\bm{\Sigma}}$, we have
\[\limsup_{T\to\infty}\frac{T^{1/3}}{\sqrt{\log T}}\left(\frac{1}{T} \E^0_b\left[\int_0^Tc(X^{\hat{Z}}_s)\d s + q_u\hat{U}_T+q_d\hat{D}_T\right]-\Val_b\right)< \infty.\]
\end{theorem}

\begin{proof}
We first consider the costs in the exploration periods. Using \cite[Chapter VI, Theorem 1.2]{MR1978607}, we first see that in one exploration cycle starting and ending in 0, the expected costs are 
\[\E^0_b[\tau_1]\int c(x)\rho_b(x)\d x,\]
with finiteness of (arbitrary) moments of $\tau_1$ under $\PP^0_b$ being demonstrated in Appendix \ref{sec:proof strategy}. Hence, the expected costs per time unit in full exploration cycles are $\int c(x)\rho_b(x)\d x$\ and the time spent in such cycles until $T$ is bounded by $S_T$. If we consider the cumulative costs until time $T$, we have to take into account 
that the last exploration cycle may be cut off at the deterministic time $T$. Putting pieces together, we can bound the expected costs in the exploration period as follows:
\begin{align*}
\E^0_b\left[\int_0^Tc(X^{\hat{Z}}_t)\d S_t\right]&\leq \E^0_b\left[\sum_{\substack{n:\tau_n\leq T\\
\text{exploration period}}}\int_{\tau_n}^{\tau_{n+1}}c(X^{\hat{Z}}_t)\d t\right] \\
&= \sum_{n \in \N_0} \E^0_b\Big[\E^0_b\Big[\int_{\tau_n}^{\tau_{n+1}} c(X_t^{\hat{Z}}) \diff{t}\Big\vert \cF^{\hat{Z}}_{\tau_n}\Big] \one_{\{\tau_n \leq T, \text{ exploration starts at } \tau_n\}} \Big] \\
&= \E^0_b[N^0_T ] \E^0_b[\tau_1] \int c(x) \rho_b(x) \diff{x}\\
&\lesssim T^{2/3} ,
\end{align*}
where we applied Lemma \ref{lem:strategy} with $N^0_T$ denoting the number of exploration periods until time $T$ and $(\mathcal{F}_t^{\hat{Z}})_{t \geq 0}$ is the filtration generated by the controlled process $\mathbf{X}^{\hat{\mathbf Z}}$.
To analyze the costs in the exploitation periods, we write $R_t\coloneqq t-S_t$ for the time spent in the exploitation periods and\textemdash again using \cite[Chapter VI, Theorem 1.2]{MR1978607}\textemdash  similarly get
\begin{align*}
\E^0_b&\left[\int_0^Tc(X^{\hat{Z}}_t)\d R_t+q_u\hat{U}_T+q_d\hat{D}_T\right]\\
&\leq \E^0_b\left[\sum_{\substack{n:\tau_n\leq T\\
		\text{exploitation period}}}\left(\int_{\tau_n}^{\tau_{n+1}}c(X^{\hat{Z}}_t)\d t+q_u(\hat{U}_{\tau_{n+1}}-\hat{U}_{\tau_{n}})+q_d(\hat{D}_{\tau_{n+1}}-\hat{D}_{\tau_{n}})\right)\right]\\
	&\leq \sum_{n \in \N_0} \E^0_b\Big[\E^0_b\Big[\int_{\tau_n}^{\tau_{n+1}} c(X_t^{\hat{Z}}) \diff{t}+q_u(\hat{U}_{\tau_{n+1}}-\hat{U}_{\tau_{n}})+q_d(\hat{D}_{\tau_{n+1}}-\hat{D}_{\tau_{n}})\Big\vert \cF^{\hat{Z}}_{\tau_n}\Big] \one_{\{\tau_n \leq T, \text{ exploit.\ starts at } \tau_n\}} \Big] \\
&\leq \E^0_b\left[\sum_{n \in \N_0} C(\hat{\colow}_{\tau_n},\hat{\coup}_{\tau_n})\E^0_b[\tau_{n+1}-\tau_n\vert \cF^{\hat{Z}}_{\tau_n} ]\one_{\{\tau_n \leq T, \text{ exploitation starts at } \tau_n\}} \right]\\
&= \E^0_b\left[\sum_{n \in \N_0} C(\hat{\colow}_{\tau_n},\hat{\coup}_{\tau_n})(\tau_{n+1}-\tau_n)\one_{\{\tau_n \leq T, \text{ exploitation starts at } \tau_n\}} \right]\\
&= \E^0_b\left[\sum_{n \in \N_0} \int_{\tau_n}^{\tau_{n+1}}C(\hat{\colow}_t,\hat{\coup}_t)\d t \one_{\{\tau_n \leq T, \text{ exploitation starts at } \tau_n\}} \right]\\
&\leq \int_0^T\E^0_b[C(\hat{\colow}_t,\hat{\coup}_t)]\d t+\max_{(\colow,\coup)\in K_{\Bd}}C (\colow,\coup)\E^0_b[\overbar{\eta}{}^1],
\end{align*}
where $\overbar{\eta}{}^1$ denotes the length of an exploitation period with maximal length (i.e., a period with reflection in $\pm B$).
On the event $\{S_t\geq mt^{2/3}\}$, we have that $(\hat{\colow}_t,\hat{\coup}_t)=(\tilde{\colow}_{mt^{2/3}},\tilde{\coup}_{mt^{2/3}})$, so that by Lemma \ref{lem:strategy} and Proposition \ref{prop:nonparam_estimator}, we have
\begin{align*}
	\E^0_b[C(\hat{\colow}_t,\hat{\coup}_t)]&\leq \max_{(\colow,\coup)\in K_{\Bd}}C (\colow,\coup)\PP^0_b(S_t< mt^{2/3})+\E^0_b[C(\hat{\colow}_t,\hat{\coup}_t)\one_{\{S_t\geq mt^{2/3}\}}]\\
	&\leq c_1t^{-1/3}+\E^0_b[C(\tilde{\colow}_{mt^{2/3}},\tilde{\coup}_{mt^{2/3}})]\\
	&\leq c_1t^{-1/3}+\Val_b+c_2\sqrt{\frac{\log (mt^{2/3})}{mt^{2/3}}}\\
	&\leq \Val_b+c_3 \frac{\sqrt{\log t}}{t^{1/3}}
\end{align*}
for certain constants $c_1,c_2,c_3$, hence
\begin{align*}
	\E^0_b\left[\int_0^Tc(X^{\hat{Z}}_t)\d R_t+q_u\hat{U}_T+q_d\hat{D}_T\right]
&\leq \int_0^T\E^0_b[C(\hat{\colow}_t,\hat{\coup}_t)]\d t+\max_{(\colow,\coup)\in K_{\Bd}}C (\colow,\coup)\E^0_b[\overbar{\eta}{}^1]\\
&\leq \Val_bT+c_4\int_0^T\frac{\sqrt{\log t}}{t^{1/3}}\d t\\
&\leq \Val_bT+c_4{\sqrt{\log(T)}}\int_0^T{t^{-1/3}}\d t\\
&\leq \Val_bT+c_5T^{2/3}{\sqrt{\log(T)}}
\end{align*}
for certain constants $c_4,c_5$.
Putting pieces together , we obtain 
\begin{align*}
&\frac{1}{T} \E^0_b\left[\int_0^Tc(X^{\hat{Z}}_s)\d s + q_u\hat{U}_T+q_d\hat{D}_T\right]-\Val_b\\
&=\frac{1}{T}\E^0_b\left[\int_0^Tc(X^{\hat{Z}}_t)\d S_t\right]+\frac{1}{T}\E^0_b\left[\int_0^Tc(X^{\hat{Z}}_t)\d R_t+q_u\hat{U}_T+q_d\hat{D}_T\right]-\Val_b\\
&\lesssim T^{-1/3}+\frac{\sqrt{\log(T)}}{T^{1/3}}.
\end{align*}
\end{proof}


\section{Data-driven controls for L\'evy processes} \label{sec:levy}
We now turn our attention to another class of non-continuous control problems. 
The first main difference is that we consider a one-sided class of problems, that is, we just consider downward controls. 
Second, we assume the underlying processes to have jumps. 
More precisely, as our driving process, we take a L\'evy process $\mathbf{X} = (X_t)_{t \geq 0}$, started in $x \in \R$ under $\PP^x$, satisfying the basic assumption
\begin{enumerate}[leftmargin=*,label=($\mathscr{H}0$),ref=($\mathscr{H}0$)] 
	\item $\mathbf{X}$ is upward regular, i.e., $$\PP^0(\inf\{t \geq 0: X_t > 0\} = 0) = 1,$$ and moreover $0 < \E^0[X_1] \eqqcolon \eta < \infty$. \label{ass: upward}
\end{enumerate}
Let us note that any Lévy process with unbounded variation (i.e., either $\X$ has a non-trivial Gaussian part or $\int_{-1}^1 \lvert x \rvert \, \Pi(\diff{x}) = \infty$) satisfies the upward regularity assumption. For a full description of upward regularity in terms of necessary and sufficient conditions, also covering a subset of Lévy processes with bounded variation, see \cite[Theorem 6.5]{kyprianou2014}.

Control problems with underlying jump processes are known to be much harder to analyze than their counterparts without jumps, see \cite{oksendal2019} for discussions and many examples. To formulate our problem, we fix a non-decreasing function $\gamma \in \mathcal{C}^2(\R)$. In contrast to the singular controls discussed in Section \ref{subsec:control_diffusion}, we now consider controls of impulse-type. These are sequences $S=\left(\tau_n,\zeta_n\right)_{n \in \N}$ of stopping times $\tau_1<\tau_2<\ldots \nearrow\infty$ and $\cF_{\tau_n}$-measurable random variables $\zeta_n$ describing the times of the interventions and the state after exercising the control, respectively.
The corresponding controlled process is given as 
\[X^S_t=X_t-\sum\limits_{n\in \N : \tau_n \leq t}(X^S_{\tau_n,-}-\zeta_n), \quad t \geq 0.\]
Here, the value at time $\tau_n$, but with the control not having taken place yet, is denoted by
$$X^S_{\tau_n,-}=X_{\tau_n}-\sum\limits_{m \in \N: m<n}(X^S_{\tau_m-}-\zeta_m).$$
In general, for processes with jumps, this quantity may deviate from both the value $X^S_{\tau_n}=\zeta_n$ at time $\tau_n$ after the control has taken place and the left limit $X^S_{\tau_n-}$. 
We can interpret $\gamma(\mathbf X^S)$ as the value of a natural resource we are managing. 
In most examples of interest, $\gamma$ has a sigmoidal form, so that (without interventions) the value is expected to grow fast whenever $\mathbf X^S$ takes moderate values, while the value grows slowly whenever $\mathbf X^S$ has either large or small values. 
The stopping times $\tau_n$ describe the times of intervention. 
From the motivating problem, it is clear that we only assume downward controls to be admissible, i.e., we assume that $X^S_{\tau_n,-}\geq \zeta_n$ for all\ $n$. 

Our aim is to find a maximizer and the corresponding value $v$ of the expected rewards without fixed transaction costs, defined by
\begin{align}\label{eq:value_levy}
	\liminf_{T\rightarrow\infty} \frac{1}{T}\E^x\Big[\sum_{n \in \N: \tau_n\leq T}\left(\gamma\left(X^S_{\tau_n,-}\right)-\gamma\left(\zeta_n\right)\right)\Big],
\end{align} 	
in the class of all admissible impulse control strategies $S=\left(\tau_n,\zeta_n\right)$. 

We will argue in Section \ref{sec: data levy} below that the main tool for solving \eqref{eq:value_levy} is the ascending ladder height process. For the reader who is not familiar with this notion, we have collected the underlying concepts and main results needed in the following in Appendix \ref{sec: appendix levy}. 
Note that \ref{ass: upward} implies $\lim_{t \to \infty} X_t = \infty$ almost surely. Let $H_t = X_{\mathsf{L}^{-1}_t}$, $t \geq 0$, be the ascending ladder height subordinator of $\X$, where $\mathsf{L} = (\mathsf{L}_t)_{t \geq 0}$ is a version of local time at the supremum and $\mathsf{L}^{-1}= (\mathsf{L}^{-1}_t)_{t \geq 0}$ is its right-continuous inverse. Note that $\mathsf{L}$ can be chosen to be continuous by upward regularity of $\X$, which entails that $t \mapsto \mathsf L^{-1}_t$ is strictly increasing and thus $\mathbf{H}$ is a strictly increasing subordinator (or, put differently, is not compound Poisson). Moreover, for any $t \geq 0$, $\mathsf{L}^{-1}_t$ is an $\F$-stopping time, where $\F = (\cF_t)_{t \geq 0}$ denotes the usual completed natural filtration of $\mathbf{X}$. Motivated by the solution technique for the associated control problem, we choose a scaling of $\mathsf{L}$ such that $\E^0[\mathsf{L}_1^{-1}] = 1$ and hence, by Wald's equality (cf.\ \cite[Corollary 2.5.2]{pena99}),
$$\E^{0}[H_1] = \E^{0}[X_1].$$
Let $(d_H, \Pi_H)$ denote drift and L\'evy measure of $\mathbf{H}$ and $\mathcal{D}(\mathcal{A}_H)$ be the domain of the extended generator $\mathcal{A}_H$ of $H$, i.e., a measurable function $f$ belongs to $\mathcal{D}(\mathcal{A}_H)$ if the exists some measurable function $g$ such that
$$f(H_t) - f(H_0) - \int_0^t g(H_s) \diff{s}, \quad t \geq 0,$$
is a local martingale. By It\={o}'s formula for semimartingales applied to $\mathbf H$, see e.g.\ Theorem I.4.57 in \cite{jacod2003}, it follows that for $f \in \mathcal{C}^2(\R)$ such that
\begin{equation} \label{eq: generator}
g_f(x) = d_Hf^\prime(x) + \int_{0+}^\infty (f(x+y) - f(x))\, \Pi_H(\diff{y}), \quad x \in \R,
\end{equation}
is well-defined that $(f(H_t) - f(H_0) - \int_0^t g_f(H_s) \diff{s})_{t \geq 0}$  is a local martingale. For such functions $f \in \mathcal{C}^2(\R)$ we set $\mathcal{A}_H f = g_f$.
We will see in Section \ref{sec: data levy} below that the auxiliary function $f(x) = \mathcal{A}_H \gamma(x)$ is the key for the solution to \eqref{eq:value_levy}. More precisely, $\mathcal{A}_H \gamma$ yields a maximum representation of the payoff that is needed to guarantee optimality of a threshold time, which can be derived from $\mathcal{A}_H \gamma$. 

\subsection{Estimating generator functionals for the ascending ladder height L\'evy process} \label{subsec: ascending}
Motivated by this observation, to implement a data-driven strategy our goal is to find an estimator of $f(x) = \mathcal{A}_H \gamma(x)$ for  an appropriate $\gamma \in \mathcal{C}^2(\R)$, based on a continuously observed trajectory $(X_t)_{0\leq t \leq T}$ of $\X$ up to some fixed time horizon $T$, with good approximation properties wrt the $\sup$-norm risk. 
Estimating $\mathcal{A}_H \gamma$ is therefore of significant applied interest and, as it will turn out, establishing bounds for the sup-norm risk provides the right tool to infer estimates for the expected regret of data-driven control strategies. For our purposes, we will need to assume that $\gamma^\prime$ is bounded, which is clearly in line with a typically sigmoidal form of $\gamma$.

In order to construct an estimator for $\mathcal{A}_H\gamma$, a first intuitive approach would be to assume that $\Pi_H$ is absolutely continuous with Lebesgue density $\pi_H$ and reconstruct a path $(H_t)_{0 \leq \mathsf{L}^{-1}_t \leq T}$ from the full observations $(X_t)_{0 \leq t \leq T}$ to develop a nonparametric estimator $(\hat{d}_H, \hat{\pi}_H)$ of $(d_H,\pi_H)$ and then analyze the plug-in estimator 
\begin{equation}\label{est alt}
\hat{\mathcal{A}_H \gamma}(x) = \hat{d}_H \gamma(x) + \int_{0+}^\infty (\gamma(x+y) - \gamma(x)) \hat{\pi}_H(y) \diff{y}, \quad x \in \R,
\end{equation}
based on convergence rates of $(\hat{d}_H, \hat{\pi}_H)$ as $T \to \infty$. An appropriate estimator for $\hat{\pi}_H$ in this scenario is given by 
$$\hat{\pi}_H(x) = \frac{1}{\mathsf{L}_T}\sum_{0 \leq t \leq \mathsf{L}_T} K_h(x - \Delta H_t) \one_{\{\Delta H_t > 0\}}, \quad x > 0,$$
for $K_h \coloneqq h^{-1} K(\cdot \slash h)$, where $h = h(T) > 0$ is some bandwidth and $K$ a high-order kernel function, see \cite{shimizu2006,shimizu2009}. 

However, under \ref{ass: upward}---even with a full record of $\X$---local time $\mathsf{L}$ cannot be observed in general since its construction is not purely path dependent, see \cite[Chapter 4]{bertoin1996}. Hence, in our framework such ansatz is hopeless, unless $\X$ is assumed to have a one sided jump structure, i.e., $\X$ is either a subordinator (increasing paths), spectrally negative (only negative jumps but non-monotone paths) or spectrally positive (only positive jumps but non-monotone paths). In the subordinator case we can simply choose $\mathsf{L}_t = t$ and hence $\mathbf{H} = \mathbf{X}$, i.e., the problem of estimating the generator functional via the plug-in estimator reduces to estimation of the Lévy measure and drift of $\X$. When $\X$ is spectrally negative, we can choose $\mathsf{L}_t = c\overbar{X}_t$, where $\overbar{X}_t = \sup_{0 \leq s \leq t} X_s$ and $c$ is some scaling factor. Then, $\mathsf{L}^{-1}_t = T_{t/c}$, where $T_{x} \coloneqq \inf\{t \geq 0: X_t > x\}$ is the first passage time of the level $x$. By exclusively negative jumps, $\X$ reaches its maxima continuously, hence  $H_t = X_{T_{t/c}} = t/c$, where $c$ must be chosen such that 
\begin{equation}\label{eq: scaling spec neg}
1 = \E^0[\mathsf{L}^{-1}_1] = \E^0[T_{1/c}],
\end{equation}
by our required scaling of local time. Thus, estimation of the generator functional in this case boils down to estimating the drift $d_H = 1/c$, which would require estimation of expected first passage times for different levels $x > 0$ and solving \eqref{eq: scaling spec neg} for $c$ with the expectation on the right hand side replaced by the constructed estimators. This is a non-trivial procedure and it is not clear how such issue should be efficiently attacked with a given dataset. For the case of spectrally positive processes, a similar issue would arise for the correct scaling of local time at the infimum. 

Thus, the only direct estimation approach other than the one we introduce below, demands restricting to a subordinator.  If its Lévy measure is finite (i.e., $\X$ must be a compound Poisson subordinator with positive drift since we require \ref{ass: upward}), it follows from Theorem 3.1 in \cite{shimizu2009} that (ignoring the drift part) the $L^2$-risk of the estimator \eqref{est alt} is of order $1\slash \sqrt{T}$. At the end of this section, we will argue that even in this much more restricted setting, our estimator---which can be applied for arbitrary jump structures---matches this performance.

Let us therefore now show how to go a more sophisticated route, exploiting the probabilistic structure of the generator functional by making use of stability results on \textit{overshoots} of L\'evy processes recently discussed in \cite{doering21}. This is in general a very natural approach for statistical inference of objects related to $\mathbf{H}$  due to its intimate connections with overshoots of $\mathbf{X}$, briefly described in the sequel. An overshoot $\cO_x$ of $\mathbf{X}$ over the level $x \geq 0$ is defined by 
$$\cO_x \coloneqq X_{T_x} - x, \quad x \geq 0,$$
where $T_x \coloneqq \inf\{t \geq 0: X_t > x\}$ is again the first hitting time of $(x, \infty)$. If we consider the spatial levels that $\mathbf{X}$ surpasses along its lifetime as time index, it can be shown that under \ref{ass: upward} the overshoot process $\bm \cO \coloneqq (\cO_t)_{t \geq 0}$ is a Feller Markov process.

Its role in revealing the characteristics of the ascending ladder height process stems from the simple observation that the closure of the range of $\mathbf{H}$ almost surely is identical to the range of the running supremum process $\overbar{X}_t \coloneqq \sup_{0 \leq s \leq t} X_s$, $t \geq 0$, and hence the overshoot process $\bm\cO^H$ of $\mathbf{H}$ is indistinguishable from $\bm\cO$. It is shown in \cite{doering21} that the unique invariant distribution of $\bm\cO$ is given by 
\begin{equation*}
\begin{split}
\mu(\diff{y}) &= \frac{1}{\E^{0}[H_1]}\Big(d_H \delta_0(\diff{y}) + \one_{(0,\infty)}(y) \Pi_H(y,\infty) \diff{y} \Big)\\
&= \frac{1}{\eta}\Big(d_H \delta_0(\diff{y}) + \one_{(0,\infty)}(y) \Pi_H(y,\infty) \diff{y} \Big), \quad y \geq  0,
\end{split}
\end{equation*}
with the second equality being a consequence of our particular scaling of local time. If we assume additionally that 
\begin{enumerate}[leftmargin=*,label=($\mathscr{H}1$),ref=($\mathscr{H}1$)] 
\item either, $d_H > 0$, or there exists $(a,b) \subset (0,\infty)$ such that $\bm{\lambda}\vert_{(a,b)} \ll \Pi_H\vert_{(a,b)}$, \label{ass: ergodic}
\end{enumerate}
it follows from Proposition \ref{prop: ov tv} that, for any $x \in \R_+$,
\begin{equation} \label{eq: tv}
\lim_{t \to \infty} \lVert \PP^x(\cO_t \in \cdot) - \mu \rVert_{\mathrm{TV}} = 0,
\end{equation}
where $\lVert \cdot \rVert_{\mathrm{TV}}$ (as before) denotes the total variation distance. 
In Proposition \ref{prop: ov tv}, conditions on the characteristics of the parent process $\X$ implying \ref{ass: ergodic} are given. These underline that most explicit L\'evy models fall into the total variation convergence scheme provided that upward regularity is satisfied, since these usually either possess a non-trivial Gaussian component or the L\'evy measure is constructed from a Lebesgue density. Finally, assuming
\begin{enumerate}[leftmargin=*,label=($\mathscr{H}2$),ref=($\mathscr{H}2$)]
\item there is $\lambda > 0$ such that $\E^{0}[\exp(\lambda H_1)] < \infty$, \label{ass: exp}
\end{enumerate}
which is true iff $\Pi\vert_{[1,\infty)}$ integrates $x \mapsto \exp(\lambda x)$, Proposition \ref{prop: ov exp} states that total variation convergence in \eqref{eq: tv} takes place at exponential rate and that $\bm\cO$ is exponentially $\beta$-mixing whenever the initial distribution integrates $\exp(\lambda \cdot)$. In particular, the stationary overshoot process is exponentially $\beta$-mixing, with $\beta$-mixing coefficient
\begin{equation} \label{eq: ov exp mix}
\beta(t) = \int_{\R_+} \lVert \PP^x(\cO_t \in \cdot) - \mu \rVert_{\mathrm{TV}} \, \mu(\diff{x}) \leq C(\lambda,\delta,\mu) \mathrm{e}^{-t/(2+\delta)}, \quad t \geq 0,
\end{equation}
for some constant $C(\lambda,\delta,\mu) \in (0,\infty)$ and arbitrary $\delta \in (0,1)$. Starting from this general setup, the fundamental observation for our purposes is that, for $\gamma \in \mathcal{C}^2(\R)$, we can rewrite \eqref{eq: generator} in terms of an integral wrt the invariant overshoot distribution $\mu$.

\begin{lemma}
For any $\gamma \in \mathcal{C}^2(\R)$ with bounded derivative we have
\[
\mathcal{A}_H \gamma(x) = \int_{\R_+} \eta \gamma^\prime(x+y) \, \mu(\diff{y}), \quad x \in \R.
\]
\end{lemma}
\begin{proof}
Note first that $\E^0[H_1] < \infty$ and boundedness of $\gamma^\prime$ guarantee that both sides of the equation are well defined. Plugging in and using Fubini we obtain for $x \in \R$,
\begin{align*}
\int_{\R_+} \eta \gamma^\prime(x+y) \, \mu(\diff{y}) &= d_H \gamma^\prime(x) + \int_{0+}^\infty \gamma^\prime(x+y) \int_{y+}^\infty \, \Pi_H(\diff{z}) \diff{y}\\
&= d_H \gamma^\prime(x) + \int_{0+}^\infty \int_{(0,z)} \gamma^\prime(x+y) \diff{y}\, \Pi_H(\diff{z})\\
&= d_H \gamma^\prime(x) + \int_{0+}^\infty (\gamma(x+z) - \gamma(x)) \, \Pi_H(\diff{z})\\
&= \mathcal{A}_H \gamma(x).
\end{align*}
\end{proof}
\begin{remark}
This formula is valid for any subordinator with finite mean.
\end{remark}

It follows from von Neumann's ergodic theorem that, for any $x \geq 0$ and $p \geq 1$,
$$\lim_{S \to \infty} \frac{1}{S} \int_0^S \eta\gamma^\prime(x+ \cO_t) \diff{t} = \mathcal{A}_H \gamma(x), \quad \text{in } L^p(\PP^\mu).$$
It is therefore natural to consider as an estimator of $f(x) = \mathcal{A}_H \gamma(x)$, based on overshoot observations $(\cO_t)_{0 \leq t \leq S}$ up to some \textit{spatial} level $S > 0$, the unbiased (under $\PP^\mu$) estimator 
$$\tilde{f}_S(x) = \frac{1}{S} \int_0^S \eta\gamma^\prime(x+ \cO_t) \diff{t}, \quad x \in \R,$$
with $\eta = \E^0[X_1] > 0$ assumed to be known (which is not a strict assumption in light of i.i.d.\ increments of $\mathbf{X}$). 
To establish convergence bounds wrt to the $\sup$-norm risk, we make use of Proposition \ref{prop:mix}.
We apply this result to the function class 
$$\mathcal{G} \coloneqq \{\eta\gamma^\prime(x+ \cdot) - \mathcal{A}_H \gamma(x): x \in \mathbb{Q} \cap D \},$$
to find a convergence rate of $1 \slash \sqrt{S}$ for the $\sup$-norm risk 
$$\mathcal{R}^D_\infty\big(\tilde{f}_{S}, f\big) \coloneqq \E^{0}\big[\big\lVert \tilde{f}_S - f\big\rVert_{D,\infty} \big],$$
for some bounded open set $D \subset \R$. The choice of evaluating the sup-norm risk wrt $\PP^0$ is somewhat arbitrary and can be replaced by $\PP^x$ for any $x \geq 0$ by spatial homogeneity of the Lévy process. We stress however that, although we make heavily use of ergodic arguments, we do not need the process to be started in the stationary overshoot distribution for our results. Similar to the proof of Corollary \ref{coro:invdens}, the key for this is Lemma \ref{lem:conv int} in conjunction with exponential ergodicity of $\bm{\mathcal{\cO}}$. 

Let us assume for the rest of the section that \ref{ass: upward} -- \ref{ass: exp} are satisfied, if not mentioned otherwise.

\begin{proposition}\label{prop: rate level}
Let $\gamma \in \mathcal{C}^2(\R)$ such that $\gamma^\prime$ is bounded.
Then there exists a constant $C_1 > 0$ such that
$$\mathcal{R}^D_\infty\big(\tilde{f}_{S}, f\big) \leq C_1\frac{1}{\sqrt{S}}.$$
\end{proposition}
\begin{proof}
By stationarity of $\bm\cO$ under $\PP^\mu$ and its exponential $\beta$-mixing property \eqref{eq: ov exp mix}, which is guaranteed given our assumptions, it follows easily (see, e.g., the proof of Proposition 2.4 in \cite{dexheimer20}) for any bounded $g$ and $t > 0$ that 
\begin{align*}
\lVert g \rVert_{\mathbb{G},t}^2 = \frac{1}{t}\mathrm{Var}\Big(\int_0^t g(\cO_s)\diff{s} \Big) &\leq 2\lVert g \rVert_\infty^2 \int_0^t \int_0^\infty \lVert \PP^x(\cO_s \in \cdot) - \mu \rVert_{\mathrm{TV}} \,\mu(\diff{x})\diff{s}\\
&\leq 2\lVert g \rVert_\infty^2 \varrho(\lambda,\delta,\mu) (2+\delta),
\end{align*}
for some $\delta \in (0,1)$. Hence, there exists a constant $\tilde{C} > 0$ such that, independently of $t > 0$, for any bounded $f,g$
\begin{equation}\label{eq: dist bound}
d_{\mathbb{G},t}(f,g) \leq \tilde{C} d_\infty(f,g).
\end{equation}
Letting $\mathcal{G} \coloneqq \{\eta\gamma^\prime(x+ \cdot)- \mathcal{A}_H \gamma(x): x \in \mathbb{Q} \cap D\}$ and using the fact that $\gamma^\prime$ is Lipschitz on the bounded set $D$ thanks to $\gamma \in \mathcal{C}^2(\R)$, it follows with Lemma \ref{lemma: covering numbers} that 
$$\mathcal{N}(\varepsilon, \mathcal{G}, d_\infty) \leq \frac{4\eta L\mathrm{diam}(D)}{\varepsilon}, \quad \varepsilon > 0,$$
where $L$ denotes the Lipschitz constant of $\gamma^\prime$ on $D$. It therefore follows that the associated entropy integral is finite, i.e.,
$$\int_0^\infty\log\mathcal N\big(u,\mathcal G, d_\infty\big)\d u < \infty,$$
and by \eqref{eq: dist bound} the same is true for the entropy integral 
$$\int_0^\infty \sqrt{\log\mathcal N\big(u,\mathcal G, d_{\G,t}\big)}\d u < \overbar{C},$$
with a constant $\overbar{C}$ independent of $t$. Since $f(x) = \mathcal{A}_H\gamma(x) = \eta \mu(\gamma^\prime(x+ \cdot))$, choosing $m_S = \sqrt{S}$ and plugging into \eqref{eq: unimom} therefore reveals that there exists a constant $C_0 > 0$ such that
\begin{equation} \label{eq: rate station}
\begin{split}
\E^\mu\Big[\sup_{x \in D} \lvert \tilde{f}_S(x) - f(x)\rvert \Big] &= \E^\mu\Big[\sup_{x \in D \cap \mathbb{Q}_+} \lvert \tilde{f}_S(x) - f(x)\rvert \Big]\\
&= \frac{1}{\sqrt{S}}\E^\mu\left[\sup_{g\in\mathcal G}|\G_S(g)|\right]\\
&\leq C_0 \frac{1}{\sqrt{S}}.
\end{split}
\end{equation}
As in the proof of Corollary \ref{coro:invdens}, we transfer the sup-norm risk bound from the stationary regime to the case when $\X$ is started in $0$. This can again be achieved utilizing exponential ergodicity of $\bm{\mathcal{O}}$. Let 
$$g(x,y) = \eta \gamma^\prime(x+y) - \mathcal{A}_H \gamma(x), \quad x,y \geq 0.$$
Then,
$$\lVert g \rVert_\infty \leq \mathfrak{B} \coloneqq \lVert \eta \gamma^\prime \rVert_\infty +\lVert \mathcal{A}_H\gamma\rVert_\infty,$$ 
which is finite by boundedness of $\gamma^\prime$. Using exponential ergodicity of $\bm{\mathcal{O}}$ as stated in Proposition \ref{prop: ov exp} and applying Lemma \ref{lem:conv int} shows that for $\delta \in (0,1)$ and $S$ large enough such that $S \geq (2+\delta) \log S$ 
\begin{align*}
&\Big\vert \E^0\Big[\sup_{x \in D} \lvert \tilde{f}_S(x) - f(x)\rvert \Big] - \E^\mu\Big[\sup_{x \in D} \lvert \tilde{f}_S(x) - f(x)\rvert \Big] \Big\vert \\
&\leq \Big\vert \E^0\Big[\sup_{x \in D}\Big\vert\frac{1}{T}\int_0^T g(x,\mathcal{O}_s)\Big\vert\Big] - \Big\vert \E^\mu\Big[\sup_{x \in D}\Big\vert\frac{1}{T}\int_0^T g(x,\mathcal{O}_s)\Big\vert\Big] \Big\vert\\
&\leq 2(2+\delta) \mathfrak{B} \frac{\log S}{S} + c(\delta)\mathfrak{B} \mathcal{R}_\lambda \exp(\lambda\cdot)(0) \frac{1}{S}\\
&\lesssim \frac{\log S}{S} + S^{-1}.
\end{align*}
Together with \eqref{eq: rate station}, this implies that
$$\mathcal{R}^D_\infty\big(\tilde{f}_{S}, f\big) \leq C_1 S^{-1/2}$$
for some constant $C_1 > 0$, by triangle inequality.
\end{proof}

Proposition \ref{prop: rate level} shows that $\tilde{f}_S$ is not only an elegant but also efficient estimator for $f = \mathcal{A}_H\gamma$, provided that we have an overshoot sample $(\cO_t)_{0 \leq t \leq S}$ available up to a fixed \textit{level} $S$. However, we observe the L\'evy process up to a fixed \emph{time} $T$ and not up to the random first passage time $T_S$. 
Our agenda therefore must be to build an estimator $\hat{f}_T$ which is $\mathcal{F}_T$-measurable and whose $\sup$-norm convergence properties can be inferred from Proposition \ref{prop: rate level}. To this end, we aim to make use of the law of large numbers for L\'evy processes. Recalling that $\lim_{T \to \infty} X_T\slash T = \E^0[X_1] = \eta$ almost surely for any starting distribution of $\mathbf X$, it follows that, for any $\varepsilon > 0$,
\begin{equation} \label{eq: deviation}
\PP^0\Big(\Big\lvert \frac{X_T}{T} - \eta \Big\rvert > \varepsilon \Big) \underset{T \to \infty}{\longrightarrow} 0.
\end{equation}
Define
\begin{align}\label{eq:tilde_f}
	\hat{f}_T(x) \coloneqq \frac{1}{X_T} \int_0^{X_T} \eta \gamma^\prime(x+\cO_t) \diff{t} \one_{(0,\infty)}(X_T), \quad x \in \R,
\end{align}
and note that, since $\{X_T > t\} \subset \{T_t \leq T\}$ for any $t \geq 0$, we have 
$$\gamma^\prime(x+\cO_t) \one_{\{t< X_T\}} = \gamma^\prime(x+\cO_t) \one_{\{T_t \leq T\} \cap \{t< X_T\}} \in \cF_T,$$
as a consequence of $\gamma^\prime(x+\cO_t)\one_{\{T_t \leq T\}} \in \cF_T$ thanks to $T_t$ being an $\F$-stopping time. Therefore, $\hat{f}_T(x) \in \cF_T$ for any $x \in \R$ as desired. As a key result, the following preparatory lemma shows that the two essential components involved in an upper bound of the sup-norm risk of $\hat{f}_T$ are indeed the rate of $\tilde{f}_{\eta T} = \tilde{f}_{\E^0[X_T]}$ and the speed of convergence in \eqref{eq: deviation}. 

\begin{lemma}\label{lemma: rate time}
Let $\gamma \in \mathcal{C}^2(\R)$ such that $\gamma^\prime$ is bounded. 
Then, there exists a constant $C > 0$ such that, for any $\varepsilon \in (0,\eta \wedge 1 \slash 2)$ and $T > 0$, we have
\begin{equation}\label{eq: rate time}
\mathcal{R}^D_\infty\big(\hat{f}_T, f\big) \leq C\Big(\frac{1}{\sqrt{\eta T}} + \frac{\varepsilon}{\eta} + \PP^0\Big(\Big\vert \frac{X_T}{T} - \eta \Big\vert > \varepsilon \Big) \Big).
\end{equation}
\end{lemma}
\begin{proof}
Let again $\mathfrak B \coloneqq \eta\lVert \gamma^\prime  \rVert_{\infty} + \lVert \mathcal{A}_H \gamma \rVert_{\infty} < \infty$. 
Then, for $C = 2\max\{1,C_1,\mathfrak B\}$, it follows by the triangle inequality and Proposition \ref{prop: rate level} that 
\begin{align*}
&\E^0 \left[\sup_{x \in D} \Big\vert\frac{1}{X_T} \int_0^{X_T} \eta \gamma^\prime(x+ \cO_t) \diff{t} - \mathcal{A}_H \gamma(x) \Big\vert\right] \\
&\quad \leq \E^0 \left[\frac{\eta T}{X_T}\sup_{x \in D \cap \mathbb{Q}} \Big\vert\frac{1}{\eta T} \int_0^{\eta T \frac{X_T}{\eta T}} \eta \gamma^\prime(x+ \cO_t) \diff{t} - \mathcal{A}_H \gamma(x) \Big\vert \,; \, \Big\{ \Big\vert \frac{X_T}{T} - \eta \Big\vert \leq \varepsilon \Big\}\right] + \mathfrak{B} \PP^0\Big(\Big\vert \frac{X_T}{T} - \eta \Big\vert > \varepsilon \Big)\\
&\quad \leq 2\E^0 \left[\sup_{x \in D \cap \mathbb{Q}} \sup_{\lvert \alpha \rvert \leq \varepsilon\slash \eta}\Big\vert\frac{1}{\eta T} \int_0^{\eta T (1+\alpha)} \eta \gamma^\prime(x+ \cO_t) \diff{t} - \mathcal{A}_H \gamma(x) \Big\vert\right] + \mathfrak{B}\PP^0\Big(\Big\vert \frac{X_T}{T} - \eta \Big\vert > \varepsilon \Big)\\
&\quad \leq 2\left(\mathcal{R}^D_\infty(\tilde{f}_{\eta T},f)+ \E^0 \left[\sup_{x \in D \cap \mathbb{Q}} \sup_{\lvert \alpha \rvert \leq \varepsilon\slash \eta}\Big\vert\frac{1}{\eta T} \int_{\eta T}^{\eta T (1+\alpha)} \eta \gamma^\prime(x+ \cO_t) \diff{t}\Big\vert\right]\right) + \mathfrak{B} \PP^0\Big(\Big\vert \frac{X_T}{T} - \eta \Big\vert > \varepsilon \Big)\\
&\quad \leq C\Big(\frac{1}{\sqrt{\eta T}} + \frac{\varepsilon}{\eta} + \PP^0\Big(\Big\vert \frac{X_T}{T} - \eta \Big\vert > \varepsilon \Big) \Big),
\end{align*}
where for the second inequality we used that, by our choice of $\varepsilon\in(0,\eta\wedge1\slash 2)$, we have $\eta T\slash X_T \leq (1- \varepsilon)^{-1} \leq 2$ on $\{\lvert X_T \slash T - \eta \vert \leq \varepsilon\}$.
\end{proof}

The following result complements results on tail asymptotics of the marginal $X_T$ for \textit{fixed} $T > 0$ of a Lévy process $\X$ with bounded jumps, which can be found in Theorem 26.1 of \cite{sato2013}, and non-asymptotic tail bounds of a Lévy process for small times $T > 0$, recently discussed in \cite{duval20}. It is a slight digression from the remainder of this section in the sense that the assumptions \ref{ass: upward}-\ref{ass: exp} are dropped in favour of bounded jumps and zero mean of $\X$. The statement is of independent interest since it gives nonasymptotic bounds on the speed of convergence of the law of large numbers for Lévy processes with bounded jumps and allows establishing optimal rates for our concrete estimation problem. 

\begin{theorem}\label{theo: bounded jumps}
Suppose that $\X$ is a non trivial zero mean Lévy process with bounded jumps and Lévy triplet $(a,\sigma^2,\Pi).$ Then, there exists $\beta > 0$ and $T(p) > 0$ for $p > 0$ such that for any $T \geq T(p)$,
$$\PP^0\Big(\lvert X_T \rvert > \sqrt{\beta T \log (T^p)}\Big) \leq 2 T^{-p/2}.$$
\end{theorem}
\begin{proof}
Let $\alpha \coloneqq \inf\{z > 0: \supp(\Pi) \subset \{x \in \R: \lvert x \rvert \leq z\}\}$ be the maximal jump size of $\X$. If $\alpha = 0$, $\X$ is a scaled Brownian motion since $\E^0[X_1] = 0$ and $\X$ was assumed non-trivial. In this case, the result follows directly from the exponential decay of tails of Brownian motion. Suppose therefore that $\alpha > 0$. We only show $\PP^0(X_T > \sqrt{\beta T \log T}) \leq T^{-p/2}$. The statement then follows by performing the same calculations for the dual process $\hat{\X} = -\X$, which also is a zero mean Lévy process with  jumps bounded in absolute value by $\alpha$. Since $\X$ has bounded jumps and zero mean, its Laplace exponent $\psi$ is well-defined on $(0,\infty)$ and given by 
$$\psi(z) \coloneqq \log \E^0[\exp(zX_1)] = \frac{\sigma^2 z^2}{2} + \int_{-\alpha}^{\alpha} (\mathrm{e}^{zx}-1-zx)\, \Pi(\diff{x}), \quad z > 0.$$
Furthermore, observe that $\psi$ is smooth with derivative 
$$\psi^\prime(z) = \sigma^2 z + \int_{-\alpha}^{\alpha} (x(\mathrm{e}^{zx}-1))\, \Pi(\diff{x}), \quad z > 0.$$
By \cite[Lemma 26.4]{sato2013}, $\psi^\prime$ is invertible on $(0,\infty)$ with strictly increasing inverse denoted by $\theta$. As in the proof of \cite[Lemma 26.5]{sato2013} it follows from 
$$z = \sigma^2 \theta(z) + \int_{-\alpha}^{\alpha} x(\mathrm{e}^{\theta(z) x} -1)\, \Pi(\diff{x}), \quad z > 0,$$
that 
$$\frac{z}{\theta(z)} \leq \sigma^2 + \mathrm{e}^{\theta(z)\alpha} \int_{-\alpha}^{\alpha} x^2 \, \Pi(\diff{x}), \quad z > 0.$$
Since $\theta(0+) = 0$, this yields  
$$\limsup_{z \downarrow 0} \frac{z}{\theta(z)} \leq \sigma^2 + \int_{-\alpha}^{\alpha} x^2 \, \Pi(\diff{x}).$$
This implies that there exists some $\varepsilon > 0$ and $\delta \geq 1$ such that for all $z \in (0,\varepsilon)$,
\begin{equation}\label{eq: subord1}
\theta(z) \geq \frac{z}{\delta\big(\sigma^2 + \int_{-\alpha}^{\alpha} x^2\, \Pi(\diff{x})\big)}.
\end{equation}
Moreover, it follows from \cite[Lemma 26.4]{sato2013} that for any $ x > 0$,
\begin{equation}\label{eq: subord2}
\PP^0(X_T > x) \leq \exp\Big(-\int_0^x \theta(z\slash T) \diff{z} \Big) = \exp\Big(-T \int_0^{x\slash T} \theta(z) \diff{z} \Big).
\end{equation}
Defining $\beta \coloneqq \delta(\sigma^2 + \int_{-\alpha}^{\alpha} x^2 \, \Pi(\diff{x}))$ and letting $T(p) > 0$ be large enough so that $\sqrt{\beta \log T^p \slash T} \in (0,\varepsilon)$ for all $t \geq T(p)$, it follows from \eqref{eq: subord1} and \eqref{eq: subord2} that indeed
\begin{align*}
\PP^0\big(X_T > \sqrt{\beta T \log (T^p)}\big) &\leq \exp\Big(-T \int_0^{\sqrt{\beta \log (T^p)\slash T}} \theta(z) \diff{z}\Big)\\
&\leq \exp\Big(-\frac{T}{\beta} \int_0^{\sqrt{\beta \log (T^p)\slash T}} z \diff{z}\Big)\\
&= T^{-p/2}.
\end{align*}
\end{proof}

With this preparation we can now investigate convergence rates of $\hat{f}_T$.

\begin{theorem}\label{theo: rate time}
Let $\gamma \in \mathcal{C}^2(\R)$ such that $\gamma^\prime$ is bounded.
\begin{enumerate}[label = (\roman*),ref = (\roman*)]
\item Suppose that $\E^0[\lvert X_1 \rvert^{p}] < \infty$ for some $p \geq 2$. Then,
$$\mathcal{R}^D_\infty\big(\hat{f}_T, f\big) \in \mathsf{O}\Big(T^{-\frac{1}{2(1+1\slash p)}}\Big).$$
In particular, if all moments of $X_1$ exist, then, for any $\varepsilon > 0$, 
$$\mathcal{R}^D_\infty\big(\hat{f}_T, f\big) \in \mathsf{O}\Big(T^{-\frac{1}{2+\varepsilon}}\Big).$$ \label{theo: rate time1}
\item Suppose that $\X$ has bounded jumps. Then, for $T$ large enough, it holds that 
$$\mathcal{R}^D_\infty(\hat{f}_T,f) \lesssim \sqrt{\frac{\log T}{T}}.$$
\end{enumerate}
\end{theorem}
\begin{proof}
\begin{enumerate}[leftmargin= *,label = (\roman*),ref = (\roman*)]
\item Since $\E^0[\lvert X_1 \rvert^{p}] < \infty$, it follows from 
the Burkholder--Davis--Gundy inequality for the càdlàg martingale $\tilde \X \coloneqq (X_t-\eta t)_{t \geq 0}$ (cf.\ \cite[Theorem VII.92] {dellacherie1980}), that there exists $C_p \in (0,\infty)$ such that
$$\E^0\Big[\Big\vert \frac{X_T}{T} - \eta \Big\vert^{p}\Big] = \frac{1}{T^{p}}\E^{0}[\lvert X_T - \eta T\rvert^{p}] \leq C_p \frac{1}{T^{p}} \E^0\big[[ \tilde \X ]_T^{p/2}\big] = C_p \mathrm{Var}(X_1)^{p/2} T^{-p/2},$$
Here, $([\tilde \X]_t)_{t \geq 0}$ denotes the quadratic variation of $\tilde \X$. Hence, by Markov's inequality, it follows that 
\begin{equation} \label{eq: rate time1}
\PP^0\Big(\Big\vert \frac{X_T}{T} - \eta \Big\vert > T^{-1\slash (2(1 + p^{-1}))} \Big) \leq C_p\mathrm{Var}(X_1)^{p/2}T^{p\slash (2(1+p^{-1}))}T^{-p/2} = C_p\mathrm{Var}(X_1)^{p/2} T^{-1/(2(1+p^{-1}))}.
\end{equation}
Plugging $\varepsilon = T^{-1\slash (2(1 + p^{-1}))}$ into \eqref{eq: rate time} and using \eqref{eq: rate time1}, we conclude that 
$$\mathcal{R}^D_\infty\big(\hat{f}_T, f\big) \in \mathsf{O}(T^{-1\slash (2(1+p^{-1}))}).$$
\item Since $(X_t-\eta t)_{t \geq 0}$ is a zero mean Lévy process with bounded jumps, it follows from Theorem \ref{theo: bounded jumps} that there exists some constant $\beta > 0$ such that for $T$ large enough 
\begin{equation} \label{eq: subord3}
\PP^0\Big(\Big\vert \frac{X_T}{T} - \eta\Big\vert > \sqrt{\frac{\beta \log T}{T}} \Big) = \PP^0\big(\lvert X_T - \eta T\vert > \sqrt{\beta T \log T}\big)) \leq \frac{2}{\sqrt{T}}.
\end{equation}
Thus, plugging in $\varepsilon = \sqrt{\log T \slash T}$ into \eqref{eq: rate time} gives the result.
\end{enumerate}
\end{proof}
\begin{remark} \label{rem: rate time}
\begin{enumerate}[leftmargin=*,label=(\roman*),ref=(\roman*)]
\item Since our exponential $\beta$-mixing assumption requires flat tails of $\Pi$ at $+\infty$ and moreover $\E^0[X_1] > 0$, the assumption of exponential moments is quite natural in our modelling framework. When jumps are bounded, \ref{ass: exp} is always satisfied. Hence, for most L\'evy processes falling into our estimation regime, we can expect a convergence rate of approximately $1\slash \sqrt{T}$.
\item One may wonder whether there was anything to gain, if in the definition of $\hat{f}_T$, we replaced $X_T$ by the running supremum $\overbar{X}_T$. In practice, this would be more natural since otherwise---at least intuitively---data $(\cO_t)_{X_T < t \leq \overbar{X}_T}$ was wasted and moreover the estimator becomes meaningless whenever $X_T \leq 0$ (which, as time progresses becomes increasingly unlikely). The construction of our estimator on the other hand is driven by analytical tractability.  However, in terms of the convergence rate of the estimator we cannot expect to gain much by working with the running supremum. This is evident from observing that Doob's maximal inequality for the submartingale $\X$ yields that for any $p > 1$ s.t.\ $X_1 \in L^p(\PP^0)$ and $T \geq 1$,
$$\lVert X_T \rVert_{L^p(\PP^0)} \leq \Vert \overbar{X}_T \Vert_{L^p(\PP^0)} \leq \frac{p}{p-1} \lVert X_T \rVert_{L^p(\PP^0)}.$$
\end{enumerate}
\end{remark}

Let us interpret this result in detail from a nonparametric angle and, as announced at the beginning of this section, compare our estimator $\hat{f}_T$ to the plug-in estimator given in \eqref{est alt} for the restricted setting of subordinators $\X$ with strictly positive drift $d_X > 0$ and absolutely continuous Lévy measure $\Pi$ with bounded density $\pi$, for which the latter can be applied.

For the subordinator case, the plug-in estimator has an $L^2$-convergence rate of $1\slash \sqrt{T}$. As shown in Theorem \ref{theo: rate time}, the overshoot estimator converges at rate $\sqrt{\log T \slash T}$ with respect to the $\lVert \cdot \rVert_\infty$-norm for any given Lévy process with bounded jumps satisfying \ref{ass: upward} and \ref{ass: ergodic} and hence in particular for any subordinator with Lévy measure having bounded support (but not necessarily bounded density since infinite jump activity is allowed). It is well-known from nonparametric  invariant density estimation of well-behaved scalar stochastic processes that, within a continuous observation scheme, the invariant density can be estimated with the parametric rate $1\slash \sqrt{T}$ wrt the $L^2$-norm. Estimation wrt the $\sup$-norm on the other hand introduces an additional $\log$-factor, increasing the optimal rate to $\sqrt{\log T \slash T}$, as, e.g., in the previously discussed case of scalar ergodic diffusions, see Theorem \ref{theo:invdens}. 

Thus, in the current nonparametric estimation context we observe the same phenomenon that the common price to be paid is an additional log-factor for optimal estimation with respect to the $\sup$-norm compared to the optimal $L^2$-rate. This also indicates that in principle, our approach to find an upper bound on the convergence rate of the overshoot estimator via Proposition \ref{prop: rate level} and Lemma \ref{lemma: rate time} for a time-dependent observation scheme is tight enough to establish the optimal convergence rate $\sqrt{\log T \slash T}$ for more general Lévy processes with unbounded jumps. This is evident from observing that the key result for the proof of Theorem \ref{theo: bounded jumps} is Lemma 26.4 from \cite{sato2013}, which relies on a Chernoff bound for the upper tail of a Lévy process at some fixed time $T$. However, interpreting this bound rigorously requires being able to tightly control the asymptotic behaviour of the inverse of the Laplace exponent's derivative, which for general Lévy processes is not possible. This is why we made use of Markov's inequality with power functions in the proof of part \ref{theo: rate time1} of Theorem \ref{theo: rate time} instead of the generic Chernoff bound. However, for more particular classes of Lévy processes with explicit Laplace exponent, an ansatz similar to Theorem \ref{theo: bounded jumps} may also provide the optimal convergence rate.

\subsection{Application}\label{sec: data levy}
We now return to the control problem described at the beginning of this section. 
In the following, we still assume \ref{ass: upward} -- \ref{ass: exp} and now present the main tool for our analysis, an auxiliary function $f$ defined via
$$f(x)\coloneqq 
\mathcal{A}_H \gamma(x),$$ 
where $\mathcal{A}_H$ denotes the extended generator of the ladder height process $\mathbf{H}$ of $\X$ as discussed in Section \ref{subsec: ascending}. Noting that when $\gamma \in \mathcal{C}^2_0(\R)$, Dynkin's formula and the fact that the values of $\X$ and $\mathbf H$ coincide at first hitting times almost surely, imply that
\[f(x)=\lim_{\varepsilon\downarrow  0}\frac{\E^x[\gamma(X_{T_{x+\varepsilon}})]-\gamma(x)}{\E^x[T_{x+\varepsilon}]},\]
this generates an intuition why this function is suitable for the analysis of problem \eqref{eq:value_levy}: using the theory of regenerative processes, see \cite{MR1978607}, the value
$\tfrac{\E^x[\gamma(X_{T_{x+\varepsilon}})]-\gamma(x)}{\E^x[T_{x+\varepsilon]}}$
coincides with the value of the $(s,S)$ impulse strategy which shifts the process back to $x=s$ whenever the process is above $S=x+\varepsilon$, so that---at least intuitively---$f(x)$ corresponds to the value of the reflection strategy in $x$. 
The usefulness of this approach for ergodic impulse control problems is demonstrated in \cite{christensen2020solution}, where most of the following results can be found. 
Some further complementing analysis is carried out in \cite{PhD_Tobias}. 
The main observation is that properties of the function $f$ determine the form of the optimal solution. 
For our considerations, we assume the following:
\begin{enumerate}[leftmargin=*,label=($\mathscr{H}3$),ref=($\mathscr{H}3$)] 
\item The function $f$ has a unique maximum $\maxlevy^\ast \in \R$, is strictly increasing on $\left(-\infty, \maxlevy^\ast\right]$ and strictly decreasing on $[\maxlevy^\ast,\infty)$. \label{ass: f max}
\end{enumerate}

\subsubsection{Solution for known processes using an auxiliary impulse control problem}
In \cite{christensen2020solution}, different classes of functions $\gamma$ are discussed that make \ref{ass: f max} hold for all L\'evy processes $\X$. 
The main idea for analysing the problem \eqref{eq:value_levy} is to introduce artificial fixed costs $K$ for each interaction, so that we are faced with a problem where we expect stationary impulse control strategies of $(s,S)$-type to be optimal. By considering the solutions for $K\searrow0$, we then obtain the value and an optimal strategy for the problem without costs. 
More precisely, for each $K\geq 0$, we define 
\begin{align*} &v(K)\coloneqq\sup_{S}\liminf_{T\rightarrow\infty} \frac{1}{T}\E^x\left[\sum_{n:\tau_n\leq T}\left(\gamma\left(X^S_{\tau_n,-}\right)-\gamma\left(\zeta_n\right)-K\right)\right]\label{vKdefinition},\end{align*}
where the supremum is taken over all admissible impulse control strategies $S=\left(\tau_n,\zeta_n\right)$. By elementary arguments, it is immediately seen that $v(K)$ is independent of the initial state $x$. 
To study the dependence on the fixed costs $K\geq 0$, let us shortly review the key results on long-term average impulse control problems. 
\begin{lemma}[ \cite{PhD_Tobias}, Theorem 4.3.6]
	\begin{enumerate}[leftmargin=*,label=(\roman*),ref=(\roman*)] 
		\item For all\ $K\geq 0$
		\begin{align*}
			v(K)=\sup_{x^*,{{\overbar x}\in \R,\,x^*<{{\overbar x}}}}\frac{\E^{x^\ast}\left[\gamma\left(X_{T_{\overbar x}}\right)\right]-\gamma\left({x^*}\right)-K}{\E^{x^\ast}[{T_{\overbar x}}]}.
		\end{align*}
		\item
		If $K>0$, then an $(s,S)$ strategy of the form 
		\[\tau_n=\inf\{t\geq \tau_{n-1}:X_t\geq \overbar x_K\},\;\;\zeta_n=x^*_K,\]
		is optimal. 
		The values $\overbar x_K$ and $x^*_K$ are given as follows: $\overbar x_K$ is the larger of the two roots of the equation $$f(x)=v(K).$$ If we denote the lower one by $\underbars x_K$, then $x^*_K$ is given as the maximizer $x^*_K=y\in [\underbars x_K, \overbar x_K]$ of $$ \frac{\E^y\big[\gamma\big(X_{T_{\overbar x_K}}\big) \big]-\gamma\left(y\right)-K}{\E^y[{T_{\overbar x_K}}]}.$$  
	\end{enumerate}
\end{lemma}
Under additional assumptions on the L\'evy process, it turns out that $x^*_K=\underbars x_K$ which simplifies the solution of the impulse control problems, but is not needed for our purposes. We now study the dependence of $v(K)$ on\ $K$.

\begin{theorem}[\cite{PhD_Tobias}, Theorem 4.3.6, 5.3.3, 5.3.4, and 5.3.5.]
	In the singular control problem \eqref{eq:value_levy}, the following holds true:
	\begin{enumerate}[leftmargin=*,label=(\roman*),ref=(\roman*)] 
		\item	$v=f(\maxlevy^\ast)\;\;\;\left(=\max_{x\in\R} f(x)\right)$;
		\item $v(K)\nearrow v$ as $K\searrow 0$;
		\item The $(s,S)$ strategies with upper threshold $\overbar x_K$ and restarting point $x^*_K$, $K>0$, are $\varepsilon$-optimal for \eqref{eq:value_levy} as $K\searrow 0$;
		\item $\underbars x_K\nearrow \maxlevy^\ast$ and $\overbar x_K\searrow \maxlevy^\ast$ as $K\searrow 0$.
	\end{enumerate}
\end{theorem}
The previous results suggest that the reflection strategy at level $\maxlevy^\ast$ is optimal in problem \eqref{eq:value_levy}. However, this strategy does not directly fall into the class of impulse control strategies we consider here, but is of (strictly) singular type. In order not to overburden the paper with technicalities, we leave out the discussion of extending the strategy space here. Note however that, due to our ergodic problem formulation, extending the control space is even not needed to obtain optimizers for \eqref{eq:value_levy}: the (non-stationary) threshold strategy with time dependent thresholds $x^*_{K_T}$ and $\overbar x_{K_T}$ (with $K_T\searrow 0$ as $T\nearrow \infty$) is optimal in the class of impulse strategies and converges with arbitrary speed in $T$ by choosing sufficiently small costs $K_T,\,T\geq0$. Therefore, the term `reflection strategy' refers to a suitably fast approximating impulse strategy in the following. 

\subsubsection{Data-driven singular controls}
The results in Section \ref{subsec: ascending} now directly lead to a method for estimating the optimal reflection boundary $\maxlevy^*$: after having observed the underlying L\'evy process for $T$ time units, we define the estimator for the auxiliary function $f$, using the estimator $\hat{f}_T$ defined in \eqref{eq:tilde_f}, and then choose
\begin{equation}\label{eq:strategy levy}
\ao\in\arg\max_{\maxlevy\in D}\hat{f}_T(\maxlevy),
\end{equation}
where $D$ is some arbitrary bounded, open neighborhood of $\theta^\ast$. The results from Section \ref{subsec: ascending} now yield that the estimated optimizer gives the optimal value  $v=f(\maxlevy^\ast)$ up to a regret of order $T^{-1/(2(1+1/p))}$ when the $p$-th moment of $\X$ exists and of order $\sqrt{\log T \slash T}$ when jumps of $\X$ are bounded. 
Indeed:

\begin{theorem}
Let $D$ be a bounded open neighborhood of $\theta^\ast$. Suppose that $\E^0[\lvert X_1 \rvert^{p}] < \infty$ for some $p \geq 2$. Then, it holds that
\[
\E^0\left[v-f(\ao)\right]\in \mathsf{O}\Big(T^{-\frac{1}{2(1+1/p)}}\Big).
\]
If $\X$ has bounded jumps, then 
\[\E^0\left[v-f(\ao)\right]\in \mathsf{O}\Big(\sqrt{\tfrac{\log T}{T}}\Big).\]
\end{theorem}
\begin{proof}
	As in the proof of Proposition \ref{prop:nonparam_estimator}, we obtain 
	\begin{align*}
		v -f(\ao)
		&\leq 2\max_{x\in D}\left|f(x)-\hat{f}_T(x)\right|,
	\end{align*}
hence
\[\E^0\left[v-f(\ao)\right]\leq 2\mathcal{R}^D_\infty(\hat{f}_T,f),\]
yielding the result by Theorem \ref{theo: rate time}.
\end{proof}


%

\section{Discussion}
The statistical questions discussed in this paper have a clear motivation coming from the analysis of data-driven strategies for natural classes of stochastic control problems. 
For underlying diffusion processes, the solutions to ergodic singular control problems from Section \ref{subsec:control_diffusion} can be written in terms of the invariant density, such that the key to developing data-driven strategies consists in replacing this quantity by a sample-based analogue.
From a statistical perspective, this is advantageous because rate-optimal estimation in this case (as opposed to, e.g., estimation of the drift coefficient) does not require an adaptive choice of the bandwidth. 
Due to the costs for reflection, the error measure to be used is the $\sup$-norm risk studied in Section \ref{subsec: diff station}. This is an interesting observation as for the\textemdash from the stochastic control perspective highly related\textemdash impulse control problem investigated in \cite{christensen20}, the $L^1$ risk had to be analysed.
The substantially more involved issue of bounding the $\sup$-norm risk of estimators is tackled by means of Proposition \ref{prop:mix}, exploiting mixture properties of the diffusion process. 
Since the focus of this paper is on the development of concrete control strategies, we have restricted the presentation in Section \ref{subsec: diff station} to a concise proof of the required upper bound (see Theorem \ref{theo:invdens}).
Once again, we refer to Section \ref{app:diff}, where we describe a completely self-contained derivation of the convergence rate for diffusions with drift $b\in\boldsymbol\Sigma_D(\beta)$ (cf.~\eqref{def:tildesigma}).  
In particular, it does not rely on any results for diffusion local time and can also be extended in a straightforward way to higher dimensions. 
We have reduced our explicit statistical investigation to the one-dimensional case solely because of the intended application to the stochastic control problem. 
Compared to the $L^1$ risk, the evaluation of the $\sup$-norm risk produces a well-known unavoidable logarithmic factor, which is also reflected in the expected regret per time unit (Theorem \ref{thm:main_}).

While the underlying diffusion processes in Section \ref{sec:diff} were assumed to have an ergodic behaviour allowing for a statistical analysis, this is not the case for the L\'evy-driven problem introduced in \ref{sec: data levy}. By considering a space-time transformation of the L\'evy process $\X$ in form of the overshoot process $\bm{\cO}$, we obtained an ergodic Markov process fitting right into our general modeling framework, which allows to express the quantity of interest for the singular control problem, $f = \mathcal{A}_H \gamma$, as an integral w.r.t.\ its invariant distribution. 
Combining a simple mean estimator based on an overshoot sample with classical results on the long-time behaviour of L\'evy processes then allowed us to construct an estimator whose performance depends on the tail-behaviour of $\X$ and yields an almost parametric $\sup$-norm estimation rate in case of light tails and the optimal nonparametric rate $\sqrt{\log T / T}$ when jumps are bounded. 

Based on this estimation procedure, we were then able to identify a data-driven singular control strategy, such that the estimated optimal reflection boundary yields an expected regret of the same order as the nonparametric estimation of the auxiliary function $f$.

In contrast to the diffusion case, in the L\'evy process framework we are not faced with an exploration vs.~exploitation problem: due to the spatial homogeneity of L\'evy processes, each controlled process carries the same information as the uncontrolled one (if we assume that the decision maker has access to the values $X_{\tau_n,-}^S$) as the decision maker can reconstruct an uncontrolled path by just undoing the controls.
Therefore, the following greedy strategy can be applied without additional losses: 
we use the (approximate) reflection controls with time-dependent boundary 
\[\ao\in\arg\max_{\maxlevy\in D}\hat{f}_T(\maxlevy)\]
for each time point $T$.

Finally, let us briefly outline the connection to related research fields.
The exploration vs.~exploitation trade-off encountered in Section \ref{sec:diff} is also well-known from the famous multi-armed bandit problem.
In this regard, it is interesting to observe that the number of boundaries to be estimated in the control problems in our context does not influence the rate of convergence. 
Up to the logarithmic factor coming from the $\sup$-norm vs.~$L^1$ risk discussed above, the rates of convergence indeed turn out to be the same for the two-sided problem studied here and the one-sided problem from \cite{christensen20}. 
This is in strong contrast to the related results for $\mathcal X$-armed bandit problems, see \cite{locatelli2018adaptivity,bubeck2011x}.

From a more applied point of view, it is furthermore of interest to compare the data-driven procedure proposed here to results obtained by using established methods from (deep) reinforcement learning.
These algorithms are very generally applicable, as they usually only require the presence of a Markovian decision process setting. 
For classical methods such as the regular Q-learning algorithm, very robust convergence results exist; however, the latter is not practicable for problems in which the state space is too large. 
In the stochastic control setting considered here, a very large state space cannot be avoided, and a natural approach for circumventing this obstacle is to treat the problem based on the Q-learning algorithm with function approximation.
In this respect, the fusion with neural networks has proven to be particularly powerful.
A mathematical theory of convergence for the resulting deep reinforcement learning procedures however is still under development.
Results from recent contributions such as \cite{fan20} are very interesting, but there remains a large gap between their theoretical assumptions and the Markov decision process framework that emerges for our concrete control problems. 
It seems practically impossible to apply their general convergence statements for deep Q-learning to our concrete setting such that one is forced to fall back on purely empirical tests of the algorithms.
By way of contrast, our statistically driven method allows for a thorough theoretical analysis and yields rules that are both interpretable and explainable.

Regarding the practical implementation, we do not give a detailed numerical comparison here as this strongly depends on the exact framework, but just mention that in our scenarios both approaches learn the optimal rule reasonably well, where the statistical approach is (not surprisingly) faster and for a longer time horizon very accurate.

\section*{Acknowledgement}
CS gratefully acknowledges financial support of Sapere Aude: DFF-Starting Grant 0165-00061B ``Learning diffusion dynamics and strategies for optimal control''. LT was supported by Research Training Group ``Statistical Modeling of Complex Systems'' funded by the German Science Foundation.

\begin{appendix}
\section{Supplementary material for Section \ref{subsec: diff station}} \label{app:diff}
As announced in the main part, we will now describe the foundations of the statistical analysis of the $\sup$-norm risk in the scalar diffusion setting in a more detailed manner.
To be specific, consider a stationary diffusion $\X$ on $\R$, with drift $b \in \bm{\Sigma}$, Lipschitz continuous and differentiable diffusion coefficient $\sigma$ such that $\underbars \nu \leq \sigma \leq \overbar \nu$ for some $0< \underbars \nu \leq \overbar \nu< \infty$, invariant density given in \eqref{eq: inv dens diff} and transition densities $(p_t)_{t > 0}$. 

In the scalar framework, it is standard to exploit the properties of \emph{diffusion local time}, and the proof of the tight variance bound in Theorem \ref{theo:invdens} (see \eqref{eq:Qbound}) actually relies on a corresponding result.
However, this approach has the obvious drawback that it does not allow for a natural extension to higher dimensions. 
The following assumptions represent an alternative framework which can be adapted straightforwardly to dimension $d\ge2$ and, as will be demonstrated in the sequel, at the same time also yields optimal results in dimension $d=1$. 

\begin{enumerate}[leftmargin=*,label=($\mathscr{A}$\arabic*),ref=($\mathscr{A}$\arabic*)]
\item 
For any compact set $\mathcal{S}\subset\R$, there exists a non-negative, measurable function $r_{\mathcal{S}} \colon (1,\infty) \to \R_+$ such that
\begin{equation} \label{cond2:uni}
\forall t > 1: \quad \sup_{x,y \in \mathcal{S}} \lvert p_t(x,y) - \rho(y) \rvert \leq r_{\mathcal{S}}(t), \quad \text{with } r_{\mathcal{S}}(t) \to 0 \text{ fulfilling } \int_1^\infty r_{\mathcal{S}}(t) \diff{t} < \infty.
\end{equation}
\label{ass: conv bound}
\item 
There exists a non-negative, measurable function $\alpha\colon (0,1] \to \R_+$ such that, for any $t \in (0,1]$,
\[
\sup_{x,y} p_t(x,y) \leq \alpha(t) \quad \text{and} \quad \int_{0+}^1 \alpha(t) \diff{t} = c_1 < \infty.
\]
\label{ass: density bound}
\end{enumerate}

For the sake of clarity, we restrict our presentation to the one-dimensional case as it is studied throughout the paper. 
As mentioned, the above framework and subsequent results admit straightforward extensions to higher dimensions. We refer to Section 2.1 in \cite{dexheimer20} for the precise definition of the assumptions in the general case.

\paragraph{Verifying \ref{ass: conv bound} and \ref{ass: density bound}}
Let us emphasize that in the context of scalar diffusions, \ref{ass: density bound} is a rather mild condition, which can often be formulated as a heat kernel bound in the literature---that is, $\alpha(t) = \mathfrak ct^{-1/2}$ for some $\mathfrak c>0$---under general conditions on the coefficients $b,\sigma$. E.g., classical results from PDE theory yield that if $b$ is bounded, transition densities $(p_t)_{t \geq 0}$ for the semigroup of $\X$ exist and fulfill
$$p_t(x,y) \leq c_1 t^{-1\slash 2} \exp\Big(-c_2\frac{(x-y)^2}{t}\Big), \quad t> 0, x,y\in \R,$$
for some universal constants $c_1,c_2 > 0$, implying the above form of $\alpha$ with $\mathfrak c = c_1$, see, e.g., equation (1.2) in \cite{sheu1991}. In our framework, $b$ is potentially unbounded but satisfies the linear growth condition $\lvert b(x) \rvert \leq \C(1 + \lvert x \rvert)$. Theorem 3.2 in \cite{qian2004} therefore yields that \ref{ass: density bound} is indeed satisfied in our setting. More precisely:

\begin{lemma} \label{lem: heat est}
There exists some constant $\mathfrak{L} > 0$ such that, for any $t \in (0,1]$, we have 
$$\sup_{x,y\in \R} p_t(x,y) \leq \alpha(t) \coloneqq \mathfrak{L}\frac{\sqrt{2\C}}{\sqrt{2\pi(1-\mathrm{e}^{-2\C t})}}.$$
In particular, $\lVert p_t \rVert_\infty < \infty$ for any $t > 0$ and $\int_{0+}^1 \alpha(t) \diff{t} < \infty$, i.e., \ref{ass: density bound} holds.
\end{lemma}

We proceed by showing how to derive \ref{ass: conv bound} from exponential ergodicity of $\X$, which holds under the given coefficient assumptions as argued below. 
Before we start with the proof, let us briefly show that $\X$ is a $T$-process, i.e., that there exist a lower semicontinuous kernel $T$ and a sampling distribution $a$ on $(0,\infty)$ such that for the sampled kernel
$$K_a(x,\cdot) \coloneqq \int_0^\infty P_t(x,\cdot) \, a(\diff{t}) \geq T(x,\cdot), \quad x \in \R.$$
This condition is quite general for specific models of ergodic Markov processes, see, e.g., the discussion in Appendix A.1 of \cite{dexheimer20}. In our model, $\X$ has the $\mathcal{C}_b$-Feller property, see, e.g., \cite[Theorem 19.9]{schilling2014}, and every point $x \in \R$ is reachable in finite time almost surely (cf.\ \cite[Proposition 1.15]{kut04}), i.e., for any initial distribution, $\PP(\tau_x < \infty) = 1$, where $\tau_x = \inf\{t \geq 0: X_t = x\}$. Hence, we can deduce from \cite[Theorem 7.1]{Tweedie1994} that, indeed, $\X$ is a $T$-process. Moreover, from almost sure reachability of any point, it is clear that $\X$ is \textit{Harris recurrent} wrt the Lebesgue measure $\lebesgue$, i.e., for any $B \in \mathcal{B}(\R)$ such that $\lebesgue(B) > 0$, we have $\PP^x(\tau_B < \infty) = 1$ for any $x \in \R$, where $\tau_B \coloneqq \inf\{t \geq 0: X_t \in B\}.$ 

\begin{proposition}\label{lem:rfunction}
$\X$ is exponentially ergodic, i.e., there exist constants $\coa,\cob>0$ such that
\begin{equation}\label{ann: exp}
\|P_t(x,\cdot)-\mu\|_{\operatorname{TV}}\le \cob V(x)\e^{-\coa t},\quad t\ge 0,x\in\R,
\end{equation}
where $V \geq 1$ is a $\mathcal{C}^2$-function, which is equal to $\exp(\gamma \lvert x \rvert)$ for $\lvert x \rvert > A.$
Moreover, \ref{ass: conv bound} holds true with 
$$r_\mathcal{S}(t)\coloneqq \cob(\lVert p_1 \rVert_\infty +\lVert \rho \rVert_\infty) \sup_{x \in \mathcal{S}} V(x) \mathrm{e}^{-\coa (t-1)}, \quad t > 1.$$
\end{proposition}
\begin{proof}
Denote by $\mathcal A$ the extended generator of $\X$ with domain $\mathcal D(\mathcal A)$, i.e., $f \in \mathcal{D}(\mathcal{A})$ if there exists a measurable function $g$ such that the process
\begin{equation} \label{eq: ext gen}
f(X_t) - f(X_0) - \int_0^t g(X_s) \diff{s}, \quad t \geq 0,
\end{equation}
is a local martingale.
By It\={o}'s formula, for any $f\in\mathcal C^2(\R)$ and
\begin{equation}\label{eq:gen diff}
g = bf^\prime + \frac{1}{2}\sigma^2 f^{\prime\prime} = \frac{1}{2 \rho}(\sigma^2 \rho f^\prime)^\prime,
\end{equation}
we have that $(f,g)$ satisfy \eqref{eq: ext gen}. Hence, $\mathcal C^2(\R)\subset\mathcal D(\mathcal A)$ and we write $\mathcal{A}f \coloneqq g$ for $f \in \mathcal{C}^2(\R)$ and $g$ given in \eqref{eq:gen diff}. Let $V \geq 1$ be a function as in the statement of the proposition. 
Using the definition of $\bm{\Sigma}$, it follows for $\lvert x \rvert > A$ that 
$$\mathcal{A}V(x) = \gamma \sigma^2(x) V(x) \Big(\mathrm{sgn}(x)\frac{b(x)}{\sigma^2(x)} + \frac{\gamma}{2} \Big) \leq -\frac{\gamma^2  \overbar{\nu}^2}{2} V(x).$$
Thus, $V$ satisfies the Lyapunov-type inequality
\begin{equation}\label{eq: lyap}
\mathcal{A}V(x) \leq -\frac{\gamma^2  \overbar{\nu}^2}{2} V(x) + \zeta, \quad x \in \R,
\end{equation}
where $\zeta = \sup_{x \in [-A,A]} \big(\lvert b(x)V^\prime(x)\rvert + \tfrac{1}{2}\lvert \sigma^2(x) V^{\prime \prime}(x)\rvert\big) < \infty$. 
Since $\X$ is trivially irreducible by ergodicity and also a $T$-process, it follows that any compact set is \textit{petite} by \cite[Theorem 5.1]{Tweedie1994} and, hence, $V$ is unbounded off petite sets (we refer to \cite{DownMeynTweedie1995} for the latter two notions). 
Moreover, for any $\Delta > 0$, the skeleton chain $\X^\Delta = (X_{n \Delta})_{n \in \N_0}$ is irreducible. 
This follows from observing that, by the linear growth restriction on the drift $b$, \cite[Theorem 3.1]{qian2004} yields that, for any $t > 0$, $P_t$ has a Lebesgue density $p_t$ such that for any $x,y \in \R$ we have $p_t(x,y) > 0$.
Thus, $\X^{\Delta}$ is irreducible wrt to the Lebesgue measure. Existence of an irreducible skeleton chain and Harris recurrence give aperiodicity of $\X$ as defined in \cite{DownMeynTweedie1995} and, hence, we conclude from \cite[Theorem 5.2]{DownMeynTweedie1995} that \eqref{eq: lyap} indeed implies exponential ergodicity in the sense of \eqref{ann: exp}.

Let us now come to verifying \ref{ass: conv bound}. Analogously to the proof of \cite[Proposition 5.1]{dalrei06}, define, for $x,y \in \R$,
\[
G_{y,\ep}(x) = \int_{-\infty}^x \int_{-\infty}^u \frac{2\rho(z)}{\sigma^2(u)\rho(u)}\Big(h_\varepsilon(y-z)-\int h_\varepsilon(y-v) \rho(v)\diff{v}  \Big) \diff{z} \diff{u},\quad \ep>0,
\]
where $h_\varepsilon \coloneqq \varepsilon^{-1}h(\cdot\slash \varepsilon)$ and $h \in \mathcal{C}^\infty(\R)$ with $h \geq 0$ is some bounded smoothing function fulfilling $\int h = 1$. 
Then, using $\sigma \geq \underbars{\nu} > 0$, it holds that $G_{y,\varepsilon} \in \mathcal{C}^2(\R)$ and hence $G_{y,\varepsilon} \in \mathcal{D}(\mathcal{A})$. 
Plugging in, we obtain from \eqref{eq:gen diff} that
\begin{equation}\label{eq:gen diff2}
\mathcal{A}G_{y,\varepsilon}(x) = h_{\varepsilon}(y-x) - \int h_{\varepsilon}(y-v) \rho(v) \diff{v}.
\end{equation}
It follows that 
\begin{equation} \label{eq:gen diff3}
P_t\mathcal{A}G_{y,\varepsilon}(x) = \int p_t(x,v) h_{\varepsilon}(y-v) \diff{v} - \int h_{\varepsilon}(y-v) \rho(v) \diff{v} \underset{\varepsilon \to 0}{\longrightarrow} p_t(x,y) - \rho(y).
\end{equation}
Moreover, by the properties of $h$, 
$P_1\mathcal{A}G_{y,\varepsilon}(x) \leq \lVert p_1 \rVert_\infty + \lVert \rho \rVert_\infty < \infty$, 
and, since $\mu$ is invariant, we have 
\begin{equation} \label{eq: gen0}
\mu(P_1\mathcal{A}G_{y,\varepsilon}) = \mu(\mathcal{A}G_{y,\varepsilon}) = 0,
\end{equation}
where the last equality is an immediate consequence of \eqref{eq:gen diff2}.
Combining \eqref{ann: exp}, \eqref{eq:gen diff3} and \eqref{eq: gen0}, it therefore follows for any $x,y \in \R$ and $t > 1$,
\begin{align*}
\lvert p_t(x,y) - \rho(y)\rvert = \lim_{\varepsilon \to 0} \big\lvert P_t\mathcal{A}G_{y,\varepsilon}(x) - \mu(P_1\mathcal{A}G_{y,\varepsilon}) \big\rvert &= \lim_{\varepsilon \to 0} \big\lvert P_{t-1}P_1\mathcal{A}G_{y,\varepsilon}(x) - \mu(P_1\mathcal{A}G_{y,\varepsilon}) \big\rvert\\
&\leq \cob(\lVert p_1 \rVert_\infty +\lVert \rho \rVert_\infty) V(x) \mathrm{e}^{-\coa(t-1)},
\end{align*}
implying the assertion.
\end{proof}

\paragraph{Implications of \ref{ass: conv bound} and \ref{ass: density bound} on the statistical analysis}
The above assumptions provide an alternative framework which is very well suited to deriving (optimal) bounds on the $\sup$-norm risk of estimators.
Inspection of the proof of Theorem \ref{theo:invdens} shows that it is based on three essential components, namely
\begin{enumerate}[leftmargin=*,label=(\arabic*),ref=(\arabic*)]
\item the verification of the exponential $\beta$-mixing property (as this allows to draw conclusions on the tail behaviour of the underlying process),\label{comp1}
\item a tight upper bound for the variance of integral functionals (see \eqref{eq:Qbound}), and\label{comp2}
\item upper bounds for the associated entropy integrals.\label{comp3}
\end{enumerate}
Point \ref{comp1} is addressed by the following auxiliary result.

\begin{lemma}[\cite{dexheimer20}, Proposition A.2]\label{lem:prop21}
If $\X$ is an ergodic $T$-process and there exist constants $C_{\mathcal S},\kappa_{\mathcal S}>0$ such that $r_{\mathcal S}$ in \eqref{cond2:uni} has the form $r_\mathcal S(t)= C_\mathcal S\e^{-\kappa_{\mathcal S}t}$, then $\X$ is exponentially $\beta$-mixing.
\end{lemma}

As a consequence, having already argued that $\X$ is a $T$-process, we obtain the exponential $\beta$-mixing property from Proposition \ref{lem:rfunction}.

\begin{lemma} \label{lem: diff mix}
$\X$ is exponentially $\beta$-mixing.
\end{lemma}

We proceed by demonstrating how point \ref{comp2} can be verified without resorting to the concept of diffusion local time and its specific properties.
It was already mentioned in Section \ref{sec:intro} that the $\beta$-mixing property in dimension $d=1$ is not quite sufficient to guarantee variance bounds of the order stated in \eqref{eq:Qbound}.
Indeed, let $\X$ be any ergodic scalar diffusion process with transition density $p_t$ and invariant measure $\d\mu=\rho\d\lebesgue$, let $f$ be a bounded function with compact support $\mathcal S\subset\R$ of Lebesgue measure $\lebesgue(\mathcal S)<1$, and fix $T\ge1$ and $D\in(0,T]$.
Referring to our application in Section \ref{sec:diff}, we require an upper bound of the form
\begin{equation}\label{eq:boundwanted}
\Var\left(\int_0^Tf(X_u)\diff{u}\right)\lesssim\lebesgue^2(\mathcal S)T.
\end{equation}
Using the Markov property and invariance of $\mu$, one obtains the representation
\[\Var\left(\int_0^Tf(X_u)\diff{u}\right)=2\mathcal I(0,D)+2\mathcal I(D,T),\]
for
\[\mathcal I(a,b)\coloneqq \int_a^b(T-v)\iint f(x)f(y)\left(p_v(x,y)-\rho(y)\right)\mu(\diff{x})\diff{y}\diff{v},\quad 0\le a<b\le T.
\]
Standard heat kernel bounds of the form $p_t(x,y)\lesssim t^{-1/2}$ imply that
\begin{equation}\label{eq:int1}
\mathcal I(0,D)\lesssim T\iint_{\mathcal S^2}f(x)f(y)\int_0^Dt^{-1/2}\diff{t}\mu(\diff{x})\diff{y}\le 2TD\lebesgue^2(\mathcal S)\|f\|_\infty^2.
\end{equation}
For bounding the second integral, the standard approach (which is known to provide tight variance bounds in dimension $d\ge2$) consists in exploiting mixing properties of $\X$.
Hence, a natural idea is to work under the assumption that $\X$ is exponentially $\beta$-mixing and fulfills the mixing bound \eqref{eq:mix}.
This gives
\begin{align}\nonumber
\mathcal I(D,T)&\le 2T\|f\|_\infty^2\int_D^T\int_{\mathcal S}\left(P_v(x,\mathcal S)-\mu(\mathcal S)\right)\,\mu(\diff{x})\diff{v}\\\label{eq:int2}
&\le 2T\|f\|_\infty^2\int_D^T\int_{\mathcal S}\left\|P_v(x,\mathcal S)-\mu(\mathcal S)\right\|_{\TV}\,\mu(\diff{x})\diff{v}\lesssim T\|f\|_\infty^2\e^{-\kappa D}.
\end{align}
Balancing the upper bounds \eqref{eq:int1} and \eqref{eq:int2} suggests to choose $D\asymp -2\log(\lebesgue(\mathcal S))\vee 1$, resulting in the final estimate
\[
\Var\left(\int_0^Tf(X_s)\diff{s}\right)\lesssim T\|f\|_\infty^2\lebesgue^2(\mathcal S)|\log(\lebesgue(\mathcal S))|.
\]
This upper bound differs by a logarithmic factor from \eqref{eq:boundwanted} and does \emph{not} allow to prove the optimal convergence rate.
By way of contrast, \ref{ass: conv bound} combined with the mild heat kernel estimates of the semigroup in \ref{ass: density bound}---which both hold in our setting by Lemma \ref{lem: heat est} and Proposition \ref{lem:rfunction}---yield results that are tight enough for proving the optimal upper bound on the convergence rate stated in Theorem \ref{theo:invdens}.

\begin{lemma}\label{lem:Qbound}
Fix some open and bounded set $D\subset\R$, and let $Q$ be a Lipschitz-continuous kernel function with $\supp(Q)=[-1/2,1/2]\subset D$.
Then, for any $x \in D$ and $h\in (0,1)$, it holds 
$$\mathrm{Var}\bigg(\frac{1}{\sqrt T}\int_0^T Q\Big(\frac{x-X_u}{h} \Big) \diff{u} \bigg) \leq C_0\|\rho\|_\infty h^2,\quad T>0.$$
\end{lemma}
\begin{proof}
A close inspection of the proof of Proposition 2.1 in \cite{dexheimer20} reveals that there exists some constant $C_1 > 0$ s.t., for any bounded $f$ with compact support $\mathcal{S}$ satisfying $\lebesgue(\mathcal{S}) < 1$ and $T > 0$,
$$\mathrm{Var}\Big(\int_0^T f(X_s) \diff{s}\Big) \leq C_1 \|\rho\|_\infty T \lVert f \rVert_\infty \lebesgue^2(\mathcal{S})\sup_{x \in \mathcal{S}} V(x).$$
Using that, for any $x \in D$, $\mathrm{supp}(Q((x-\cdot)/h)) = x - h\cdot \mathrm{supp}(Q) \subset D - h\cdot \mathrm{supp}(Q)$ and the right hand side is a bounded domain, it follows that $\sup_{y \in \mathrm{supp}(Q(x-\cdot)\slash h)} V(y) \leq C_2 < \infty$ with $C_2$ being a constant independent of $x \in D$. 
Moreover, $\lebesgue(\mathrm{supp}(Q(x-\cdot\slash h))) = h$, giving the assertion. 
\end{proof}

Note that Proposition \ref{lem:rfunction} (in combination with Lemma \ref{lem:prop21}) and Lemma \ref{lem:Qbound} address items \ref{comp1} and \ref{comp2}, respectively, required for the analysis of the $\sup$-norm risk.
Finally, concerning point \ref{comp3}, the following lemma provides standard bounds on covering numbers on function classes related to kernel estimation procedures wrt the norms appearing in Proposition \ref{prop:mix}. 
By a slight abuse of notation, we do not distinguish notationally between the $\sup$-norm on $\R$ and the function space $\mathcal{B}_b(\R)$.

\begin{lemma}[Lemma C.1 in \cite{dexheimer20}]\label{lemma: covering numbers}
Let $D\subset\R$ be a bounded set, $Q\colon\R\to\R$ a Lipschitz-continuous function with Lipschitz constant $L>0$, fix $h\in(0,1)$, and denote
\[\mathcal G\coloneqq \left\{Q\left(\frac{x-\cdot}{h}\right)-\int Q\left(\frac{x-\cdot}{h}\right)\diff{\mu}:\ x\in D\cap\mathbb Q\right\}.\]
If, for some $\mathfrak V>0$,
\[\Var\left(\frac{1}{\sqrt t}\int_0^tg(X_s)\diff{s}\right)\le \mathfrak V^2\|g\|_{L^2(\mu)}^2,\quad g\in\mathcal G,t>0,\]
then there exist constants $\mathfrak{A}, \upsilon > 0$ such that, for any $\varepsilon > 0$ and $t > 0$,
$$\mathcal{N}(\varepsilon, \mathcal{G}, d_{\mathbb{G},t}) \leq \left(\frac{\mathfrak{AV}}{\ep}\right)^\upsilon
\quad\text{ and }\quad
\mathcal{N}(\varepsilon, \mathcal{G},d_\infty) \leq \frac{4L\mathrm{diam}(D)}{\varepsilon h}.$$
\end{lemma}

\section{Proof of Lemma \ref{lem:strategy}} \label{sec:proof strategy}
Before we start with the proof, we need some preparatory remarks. The length of the first exploration period is given by
$$\tau_1 = \inf\{t \geq \sigma_1: X_t = 0\} = \sigma_1 + T_0 \circ \theta_{\sigma_1},$$
with $\sigma_1 = \inf\{t \geq 0: X_t = -B\} \vee \inf\{t \geq 0: X_t = B\}$, $T_a = \inf\{t \geq 0: X_t = a\}$ for $a \in \R$ and $(\theta_t)_{t \geq 0}$ denoting the transition operators of the Markov process $\X$. 
We will need that $\E^0_b[\tau_1^3] < \infty$. 
By the strong  Markov property and the fact that $\PP^0_b(\sigma_1 < \infty)$ by point recurrence of the ergodic diffusion $\mathbf{X}$, we obtain 
$$\E^0_b[\tau_1^3] \leq 4\big(\E^0_b[T_B^3] + \E^0_b[T_{-B}^3]\big) + 4\big(\E^B_b[T_0^3] +\E^{-B}_b[T_0^3]\big),$$
and thus finiteness of the third moment of $\tau_1$ boils down to finiteness of the third moment of first hitting times of the diffusion. From \cite[Section 5]{loecherbach2011}, see also \cite{balaji2000}, it is known that if the diffusion coefficient $\sigma$ is bounded and there exist $r, M_0 > 0$ such that 
\begin{equation}\label{eq: bound mom hitting}
-\frac{xb(x)}{\sigma^2(x)} \geq r, \quad \forall \lvert x \rvert > M_0,\end{equation}
then $\E^x[T_a^n] < \infty$ for all $n < r + 1/2$. 
In our setting, boundedness of $\sigma$ is satisfied with $0<\underbars{\nu} \leq \sigma(x) \leq \overbar{\nu} < \infty$, and we have 
$$\mathrm{sgn}(x) \frac{b(x)}{\sigma^2(x)} \leq - \gamma, \quad \forall \lvert x \rvert > A,$$
for some constants $\gamma, A > 0$. Thus, for $\lvert x \rvert > A \vee r\slash \gamma = M_0$, \eqref{eq: bound mom hitting} is fulfilled, which implies that for any $n \in \N$ and $a,x \in \R$, $\E^x[T_a^n] < \infty$. It is worth noting that \cite{loecherbach2011} demonstrate how the existence of moments of hitting times is initimately connected to polynomial ergodicity of a diffusion and hence the existence of arbitrarily large hitting time moments can be regarded as a consequence of exponential ergodicity of $\X$ under the given assumptions. In particular $\E^0_b[\tau_1^3] < \infty$ as desired.

We will also make use of the following simple observation.
\begin{lemma}\label{lem: var}
Let $X$ be a random variable taking values in $[1,\infty)$ almost surely and suppose that
$$(0,2) \ni \beta \mapsto \E[X^\beta]$$
is differentiable with $\tfrac{\uppartial}{\uppartial \beta} \E[X^\beta] = \E[\tfrac{\uppartial}{\uppartial{\beta}} X^\beta].$ Then, the function 
$$[0,1] \ni \alpha \mapsto \mathrm{Var}(X^\alpha)$$
is increasing.
\end{lemma}
\begin{proof}
	By the smoothness assumptions, we have
	\begin{align*}
	\frac{\uppartial}{\uppartial \alpha} \mathrm{Var}(X^\alpha) = 2\alpha \E[X^{2\alpha} \log X]- 2\alpha\E[X^\alpha]\E[X^{\alpha} \log X] = 2\alpha \mathrm{Cov}(X^\alpha, f(X^\alpha)),
	\end{align*}
with $[1,\infty) \ni x \mapsto f(x) = \alpha^{-1}x\log x$. Since $f$ is increasing, it follows that $\mathrm{Cov}(X^\alpha, f(X^\alpha)) \geq 0$, proving the assertion.
\end{proof}

We are now ready to carry out the proof of Lemma \ref{lem:strategy}.
\begin{proof}[Proof of Lemma \ref{lem:strategy}]
We start with some necessary notation. Let $\eta^{j}_n$, $n \in \N$, be the length of the $n$-th exploration period for $j = 0$ and the length of the $n$-th exploitation period for $j = 1$. In particular, $\eta^0_1 = \tau_1$ and $(\eta^0_n)_{n=2,3,\ldots}$ is an i.i.d.\ family of random variables under $\PP^0_b$. Define also $\tau^j_n \coloneqq \sum_{i=1}^n \eta^j_i$ as the length of the first $n$ exploration/exploitation periods, thus 
$$\tau_n = \tau^0_{n-\sum_{i=1}^n c_i} + \tau^1_{\sum_{i=1}^n c_i},$$
where $\tau_0^0 = \tau^1_0 = 0$. Finally, denote $N_t \coloneqq \inf\{n \in \N : \tau_n > t\}$ and the  number of exploration/exploitation periods starting before time $t \geq 0$, $N_t^j = K^j_{N_t}$, where 
$$K^{j}_n \coloneqq \#\{i\leq n: c_i = j\}, \quad n \in \N, j \in \{0,1\}.$$
Clearly, $\tau^0_{N_t^0-1} \leq S_t \leq \tau^0_{N_t^0}$ for any $t \geq 0$, and thus 
$$\frac{\tau^0_{N_t^0-1}}{N^0_t} \leq \frac{S_t}{N^0_t} \leq \frac{\tau^0_{N_t^0}}{N^0_t}.$$
By construction and the strong law of large numbers, both the left-hand and the right-hand side $\PP^0_b$-a.s.\ converge to $\E^0_b[\eta^0_2]$ and hence 
\begin{equation}\label{eq:constr1}
	\frac{S_T}{N^0_T} \underset{T \to \infty}{\longrightarrow} \E^0_b[\eta^0_2], \quad \PP^0_b\text{-a.s.}.
\end{equation}
Let now $\underbars{\eta}^j,\overbar{\eta}^j$ be such that 
$$\underbars{\eta}^j \leq_{\text{st}} \eta^j_n \leq_{\text{st}} \overbar{\eta}^j, \quad n \in \N, j \in \{1,2\},$$
where $\leq_{\text{st}}$ denotes stochastic ordering. 
Existence of such random variables is clear in case $j=0$, and for $j=1$ we can take $\underbars{\eta}^1$ to be a cycle length when reflecting in $\pm 1/B$ and $\overbar{\eta}_1$ be a cycle length when reflecting in $\pm B$. Choosing $\underbars{\eta} = \underbars{\eta}^0 \wedge \underbars{\eta}^1$ and $\overbar{\eta} = \overbar{\eta}{}^0 \vee \overbar{\eta}{}^1$ we have 
$$\underbars{\eta} \leq_{\text{st}} \eta^j_n \leq_{\text{st}} \overbar{\eta}, \quad n \in \N, j \in \{1,2\},$$
Let now $(\underbars{\eta}{}_n)_{n \in \N}$ and $(\overbar{\eta}_n)_{n \in \N}$ be i.i.d.\ copies of $\underbars{\eta}$ and $\overbar{\eta}$, resp., where by resorting to a coupling argument if needed, we may assume wlog that $\underbars{\eta}{}_n \leq \eta^j_n \leq \overbar{\eta}_n$, $\PP^0_b$-a.s.\ for all $n \in \N$, $j \in \{0,1\}$ (see, e.g., \cite[Section 3]{thorisson1995coupling}). Defining 
$$\underbars{N}{}_t = \inf\Big\{n \in \N: \sum_{i=1}^n \underbars{\eta}{}_i > t \Big\}, \quad \overbar{N}_t = \inf\Big\{n\in \N: \sum_{i=1}^n \overbar{\eta}{}_i > t \Big\},$$
it follows that 
$$\overbar{N}_t \leq N_t \leq \underbars{N}{}_t, \quad \PP_b^0\text{-a.s.},$$
With the standard renewal theorem we have
$$\lim_{t \to \infty} \frac{\underbars{N}{}_t}{t} = \frac{1}{\E^0_b\big[\underbars{\eta}\big]}, \text{ and } \lim_{t \to \infty} \frac{\overbar{N}_t}{t} = \frac{1}{\E^0_b\big[\overbar{\eta}\big]}, \quad \PP^0_b\text{-a.s.\ and in } L^1(\PP^0_b).$$
By construction, we have 
$$\frac{N^0_t}{t^{2/3}} \leq \overbar{M}\Big(\frac{N_t}{t}\Big)^{2/3} \leq \overbar{M}  \Big(\frac{\underbars{N}_t}{t}\Big)^{2/3}.$$
Since, by Jensen's inequality and the above,
$$\E^0_b\Big[\Big(\frac{\underbars{N}{}_t}{t}\Big)^{2/3}\Big] \leq \E^0_b\Big[\frac{\underbars{N}_t}{t}\Big]^{2/3} \underset{t \to \infty}{\longrightarrow} \E^0_b\big[\underbars{\eta}\big]^{-2/3},$$
it follows that
$$\limsup_{T \to \infty} T^{-2/3}\E^0_b[N^0_T] \leq \overbar{M} \E^0_b\big[\underbars{\eta}\big]^{-2/3} \eqqcolon M,$$	
which establishes the second part of the assertion.
For the first part, consider the uncontrolled diffusion process $\mathbf{X}$ and let 
$$\breve{\tau}_n = \begin{cases} 0, &\text{if } n = 0\\ \inf\{t \geq \sigma_n : X_t = 0\}, &\text{if } n \in \N, \end{cases}$$
where 
$$\sigma_n = \inf\{t \geq \breve{\tau}_{n-1} : X_t = B\} \vee \inf\{t \geq \breve{\tau}_{n-1} : X_t = -B\}, \quad n \in \N.$$
By the strong Markov property, $(\breve{\tau}_n)_{n \in \N}$ are i.i.d.\ and if we denote $\breve{N}_t = \inf\{n \in \N: \breve{\tau}_n > t\}$ for $t \geq 0$, then $(\breve{N}_t)_{t \geq 0}$ is a  renewal process with increment distribution $\breve{\tau}_1$. Furthermore, we define in analogy to the controlled case above $\breve{\eta}^i_n \coloneqq \breve{\tau}_{C^i_n} - \breve{\tau}_{C^i_n - 1}$ for $i = 0,1$, where $C^0_n \coloneqq \min\{m \in \N: \sum_{i=1}^m (1-c_i) \geq n\}$ and $C^1_n \coloneqq \min\{m \in \N: \sum_{i=1}^m c_i \geq n\}$, and $\breve{\tau}{}^j_n \coloneqq \sum_{i=1}^n \breve{\eta}{}^j_i$ for $j = 0,1$. Finally, let $\breve{N}{}^i_t = K^i_{\breve{N}_t}$ for $t \geq 0$. Clearly,  
$$\breve{\eta}^0_n \overset{\text{d}}{=} \eta^0_n \quad \text{and} \quad \breve{\eta}^1_n \geq_{\text{st}} \eta^1_n$$
for any $n \in \N$, which yields 
\begin{equation}\label{eq: st dom time}
\breve{S}_t \leq_{\text{st}} S_t, \quad t \geq 0,
\end{equation}
where 
\begin{equation} \label{eq: time expl uncontrolled}
\breve{S}_t = \int_0^t \sum_{n = 1}^{\breve{N}{}^0_t} \one_{[\breve{\tau}_{C^0_n-1}, \breve{\tau}_{C^0_n}]}(s) \diff{s} \in \big[\breve{\tau}{}^0_{\breve{N}{}^0_t-1}, \breve{\tau}{}^0_{\breve{N}{}^0_t} \big], \quad t \geq 0.
\end{equation}
By the standard renewal theorem and the properties of $(c_n)_{n \in \N}$ we have
$$\frac{\breve{N}{}^0_t}{t^{2/3}} \geq \Big(\frac{\breve{N}_t}{t} \Big)^{2/3} \underset{t \to \infty}{\longrightarrow} \E^0_b[\breve{\tau}_1]^{-2/3}, \quad \PP^0_b\text{-a.s.}$$
which, on account of the fact that by combining the strong law of large numbers and \eqref{eq: time expl uncontrolled} we have
$$\frac{\breve{S}_T}{\breve{N}^0_T} \underset{T \to \infty}{\longrightarrow} \E^0_b[\breve{\tau}_1], \quad \PP^0_b\text{-a.s.},$$
shows that 
$$\liminf_{t \to \infty} \frac{\breve{S}_t}{t^{2/3}} \geq \E^0_b[\breve{\tau}_1]^{1/3} \eqqcolon \tilde{m}, \quad \PP^0_b\text{-a.s..}$$
Consequently, an application of Fatou's lemma provides 
$$\liminf_{t \to \infty} \E^0_b\Big[\frac{\breve{S}_t}{t^{2/3}} \Big] \geq  \tilde{m}.$$
Thus, choosing $m = \tilde{m}/2$, there exists $\varepsilon \in (0,\tilde{m}/2)$ such that for any $T>0$ large enough we have 
$$\E^0_b[\breve{S}_T] \geq T^{2/3}(m+\varepsilon),$$
which, together with Markov's inequality and \eqref{eq: st dom time}, yields
\begin{align*}
\PP^0_b(S_T \leq m T^{2/3}) \leq \PP^0_b(\breve{S}_T \leq m T^{2/3})&\leq \PP^0_b(\E^0_b[\breve{S}_T] - \breve{S}_T \geq \varepsilon T^{2/3})\\
&\leq \PP^0_b\big(\E^0_b\big[\breve{\tau}{}^0_{\breve{N}{}^0_T-1}\big] - \breve{\tau}^0_{\breve{N}{}^0_T} \geq \varepsilon T^{2/3}\big)\\
&\leq \varepsilon^{-2} T^{-4/3} \E^0_b\Big[\big(\breve{\tau}{}^0_{\breve{N}{}^0_T} -\E^0_b[\breve{\tau}{}^0_{\breve{N}{}^0_T-1}\big]\big)^2\Big] \\
&\leq 2\varepsilon^{-2} T^{-4/3}\big( \mathrm{Var}_{\PP^0_b}\big(\breve{\tau}{}^0_{\breve{N}{}^0_T}\big) + \E^0_b\big[\big(\breve{\tau}_1)^2\big]\big).
 \end{align*}
Thus, to show that $\PP^0_b(S_T \leq mT^{2/3}) \lesssim T^{-1/3}$ it is enough to establish that 
\begin{equation} \label{eq: conc expl}
\mathrm{Var}_{\PP^0_b}\big(\breve{\tau}^0_{\breve{N}^0_T}\big) \lesssim T.
\end{equation}
To this end, note first that for any $n \geq 2$ and $T \geq 0$ we have 
\begin{align*}
\{\breve{N}^0_T \leq n\} &= \bigcup_{k=1}^{n} \bigcup_{l\geq k} \big( \{\breve{N}^0_T = k\} \cap \{\breve{N}_T = l\}\big)\\
&= \bigcup_{k=1}^{n} \bigcup_{l \geq k} \big\{\breve{\tau}{}^0_k + \breve{\tau}{}^1_{l-k} > T, \breve{\tau}_{l-1} \leq T, \sum_{i=1}^{l-1} (1-c_i) \in \{k-1,k\}\big\} 	\\
&\eqqcolon \bigcup_{k=1}^{n} \bigcup_{l \geq k} A_{k,l}.
\end{align*}
By construction, $(\breve{\eta}{}^0_n)_{k \geq n+1} \indep \mathcal{G}_{n}$, where
$$\mathcal{G}_{n}\coloneqq \sigma\big(\{\breve{\eta}^0_1,\ldots,\breve{\eta}{}^0_{n}\} \cup \{\breve{\eta}^1_i : i < C^0_n\}\big)$$
is the $\sigma$-algebra generated by the lengths of the first ``exploration/exploitation'' periods up to the $n$-th ``exploration'' period of the uncontrolled process $\mathbf{X}$. Clearly, $A_{k,l} \in \mathcal{G}_{n}$ for any $k=1,\ldots,n$ and $l \geq k$, which shows that $\breve{N}^0_T$ is a $(\mathcal{G}_n)_{n \in \N}$ stopping  time.  Hence, we can use Wald's second moment identity, see \cite[Propostion A10.2]{MR1978607}, to obtain  
$$\mathrm{Var}_{\PP^0_b}\big(\breve{\tau}{}^0_{\breve{N}^0_T} \big) = \mathrm{Var}_{\PP^0_b}\Big(\sum_{i=1}^{\breve{N}{}^0_T} \breve{\eta}{}^0_i \Big) = \E^0_b[\breve{\tau}_1]^2 \cdot \mathrm{Var}_{\PP^0_b}(\breve{N}^0_T) + \mathrm{Var}_{\PP^0_b}(\breve{\tau}_1) \cdot \E^0_b[\breve{N}{}^0_T].$$
Since by Jensen's inequality and the renewal theorem 
$$\limsup_{T \to \infty } T^{-2/3} \E^0_b[\breve{N}{}^0_T] \leq \limsup_{T \to \infty}\overbar{M}T^{-2/3}\E^0\big[\breve{N}_T^{2/3}] \leq \overbar{M}\limsup_{T \to \infty} \big(\E^0[\breve{N}_T]/T\big)^{2/3} < \infty,$$
\eqref{eq: conc expl} will follow if we can prove that $\mathrm{Var}_{\PP^0_b}(\breve{N}^0_T) \lesssim T$. By classical renewal arguments, cf.\ \cite[Proposition 6.3]{MR1978607}, we have 
\begin{equation} \label{eq: conv var}
\lim_{T \to \infty} \frac{\mathrm{Var}_{\PP^0_b}(\breve{N}_T)}{T} = \frac{\mathrm{Var}_{\PP^0_b}(\breve{\tau}_1)}{\E^0_b[\breve{\tau}_1]^3},
\end{equation}
and by construction it follows that
\begin{align*}
\mathrm{Var}_{\PP^0_b}(\breve{N}{}^0_T) &= \E^0[(\breve{N}{}^0_T)^2] - \E^0[\breve{N}{}^0_T]^2\\
&\leq \E^0\big[\big((\breve{N}_T)^{2/3} + \mathfrak{d}\big)^2] - \E^0\big[(\breve{N}_T)^{2/3}]^2\\
&\lesssim \mathrm{Var}_{\PP^0_b}\big((\breve{N}_T)^{2/3}\big) + \E^0_b\big[(\breve{N}_T)^{2/3}\big].
\end{align*}
By Jensen's inequality, we have $\E^0_b\big[(\breve{N}_T)^{2/3}\big] \lesssim T$, and \eqref{eq: conv var} combined with Lemma \ref{lem: var} yields $\mathrm{Var}_{\PP^0_b}\big((\breve{N}_T)^{2/3}\big) \lesssim T$ (note here that $\E^0_b[\tfrac{\uppartial}{\uppartial \beta} (\breve{N}_T)^\beta] \leq 2\E^0_b[(\breve{N}_T)^3] < \infty$ for $\beta \in (0,2)$ and thus we may differentiate under the integral to obtain $\tfrac{\uppartial}{\uppartial \beta}\E^0_b[(\breve{N}_T)^\beta] = \E^0_b[\tfrac{\uppartial}{\uppartial \beta} (\breve{N}_T)^\beta]$ as needed). Thus, $\mathrm{Var}_{\PP^0_b}(\breve{N}{}^0_T) \lesssim T$, which shows that indeed $\PP^0_b(T^{-2/3}S_T \leq m)\lesssim T^{-1/3}$.
\end{proof}

\section{A brief summary of L\'evy processes and their overshoots} \label{sec: appendix levy}
This section is devoted to giving a (very) brief summary of L\'evy processes and their overshoots in order to keep the article reasonably self-contained. Talking about overshoots quite naturally guides us into fluctuation theory of L\'evy processes, which is based on a rigorous treatment of excursions of L\'evy processes from its maximum and minimum. For an extensive textbook treatment of fluctuation theory, we refer to \cite{kyprianou2014} with basic properties of overshoots being discussed in Chapter 5. 
A general account on L\'evy processes is given in the monographs \cite{bertoin1996} and \cite{sato2013}.

We consider a L\'evy process $\X$ with underlying natural filtration $\F = (\cF_t)_{t \geq 0}$ augmented in the usual way, which is equipped with a family of probability measures $(\PP^x)_{x \in \R}$ such that $(\X,\F,(\PP^x)_{x \in \R})$ is a Markov process. Thus, $\X$ has c\`adl\`ag paths, almost surely starts in $x$ under $\PP^x$, has stationary and independent increments under $\PP^0$ and its semigroup $(P_t)_{t \geq 0}$ is given by 
$$P_t(x,B) = \PP^x(X_t \in B) = \PP^0(X_t + x \in B), \quad x \in \R, B \in \mathcal{B}(\R).$$
From the last equality, it follows that $\X$ is \textit{spatially homogeneous}, i.e. $\{\X+x,\PP^0\} \overset{\mathrm{d}}{=} \{\X,\PP^x\},$ and one easily derives that $\X$ is a Feller process, that is, if we denote by $\mathcal{C}_0(\R)$ the space of continuous functions $f \colon \R \to \R$ vanishing at infinity, it holds that $P_t \mathcal{C}_0(\R)\subset \mathcal{C}_0(\R)$ for any $t \geq 0$ and moreover $P_t f \to f$ strongly as $t \to 0$.

While the Fellerian nature of L\'evy processes is still fairly general from a Markovian perspective, it is the spatial homogeneity which gives rise to a quite unique and powerful theory. 
The fundamental starting point to the analysis of L\'evy processes is the L\'evy--Khintchine formula, which identifies the characteristic function of the marginals of the process and hence uniquely describes the complete process in terms of a \textit{characteristic triplet} $(a,\sigma,\Pi)$, where $a \in \R$, $\sigma \geq 0$ and $\Pi$ is a measure on $\R$ (called L\'evy measure) having no atom at $0$ and being such that $\int_{\R} 1 \wedge x^2 \, \Pi(\diff{x}) < \infty$. 
The L\'evy--Khintchine formula then states that $\E^0[\exp(\mathrm{i}\theta X_t)] = \exp(t\Psi(\theta))$, $\theta \in \R$, $t \geq 0$, with the \textit{characteristic exponent} $\Psi$ satisfying 
$$\Psi(\theta) = \mathrm{i}a \theta - \frac{\sigma^2\theta^2}{2} + \int_{\R} \big(\mathrm{e}^{\mathrm{i}\theta x} - 1 - \mathrm{i}\theta x \one_{[-1,1]}(x)\big)\, \Pi(\diff{x}).$$
On the level of processes, the L\'evy--Khintchine representation can be translated into a partition of the process into a linear Brownian motion $(at + \sigma B_t)_{t \geq 0}$ and an independent pure jump process characterized by $\Pi$, which itself decomposes into an independent compound Poisson process and a zero mean $L^2$-martingale with infinitely many jumps bounded by $1$ on any finite time interval. If $\int_1^\infty x \, \Pi(\diff{x}) < \infty$, it follows from the L\'evy--Khintchine representation that the first moment of $X_t$ is finite for any $t \geq 0$ and $\E^0[X_t] = t\eta$ with 
$$\eta = \E^0[X_1] = a + \int_{\R\setminus[-1,1]} x \, \Pi(\diff{x})$$
determining the long-time behaviour of $\X$ in the sense that 
\begin{enumerate}[leftmargin=*,label=(\roman*),ref=(\roman*)]
\item $\eta > 0 \implies \lim_{t \to \infty} X_t = \infty$, $\PP^0$-a.s.;
\item $\eta < 0 \implies \lim_{t \to \infty} X_t = -\infty$, $\PP^0$-a.s.;
\item$\eta = 0 \implies \limsup_{t \to \infty} X_t = -\liminf_{t \to \infty} X_t= \infty$, $\PP^0$-a.s..
\end{enumerate}
and 
$$\lim_{t \to \infty} \frac{X_t}{t} = \eta, \quad \PP^0\text{-a.s.}.$$
With this basic preparation on the characteristics of L\'evy processes, let us now come to their fluctuation theory, with a certain emphasis  on the so called Wiener--Hopf factorization. This commands a discussion of the ascending ladder height process, which dominates our analysis of data-driven solutions to impulse control problems associated to L\'evy processes. This process is derived from the local time at the supremum $\mathsf{L} = (\mathsf{L}_t)_{t \geq 0}$, which is a stochastic process measuring the time that $\X$ spends at its running supremum $\overbar{X}_t = \sup_{0 \leq s \leq t} X_s$, $t \geq 0$. Its construction is based on the observation that $\mathbf Y= (\overbar{X}_t- X_t)_{t \geq 0}$, which can be interpreted as the process obtained from reflecting $\X$ at its supremum, is a strong Markov process and hence one can define $\mathsf{L}$ as the local time at $0$ for $\mathbf{Y}$, which means that $\mathsf L$ is an additive functional of $\mathbf{Y}$ which almost surely increases on the closure of $\{t \geq 0 : Y_t = 0 \} = \{t \geq 0 : X_t = \overbar{X}_t \}$. In case that $\X$ is \textit{upward regular}, i.e., for $T_0 \coloneqq \inf\{t\geq 0: X_t > 0\}$ we have $\PP^0(T_0 = 0) = 1$, $\mathsf{L}$ can be constructed as a process with almost surely continuous paths, which entails that its right-continuous inverse $\mathsf{L}^{-1}_t = \inf\{s \geq 0: L_s > t\}$, $t \geq 0$, is almost surely \textit{strictly} increasing. In this case, the ascending ladder height process $\mathbf H = (H_t)_{t \geq 0}$, defined by 
$$H_t = \begin{cases} X_{\mathsf{L}^{-1}_t}, & \text{ if } 0\leq t < \mathsf{L}_\infty\\ \infty, & \text{ if } t \geq \mathsf{L}_\infty,\end{cases}$$
is a \textit{killed} subordinator, strictly increasing up to its lifetime $\mathsf{L}_\infty$, i.e., $\mathbf H$ is equal in law to a strictly increasing L\'evy process, which is sent to the cemetery state $\infty$ after an independent exponentially distributed random time with expectation $\E^0[\mathsf{L}_\infty]$.  If $\limsup_{t \to \infty} X_t = \infty$, which as seen before is guaranteed if $\E^0[\lvert X_1 \rvert] < \infty$ and $\eta = \E^0[X_1] \geq 0$, it follows that $\mathsf{L}_\infty = \infty$ almost surely and hence $\mathbf H$ is unkilled. Moreover, when $\eta > 0$, which is the setting that we will be working with in this paper, it holds that $0 < \E^0[H_1] < \infty$ as well, which can be deduced from \eqref{eq:vigon} below. It is important to note that $\mathsf{L}$ is only characterized uniquely up to a multiplicative constant and hence the definition of $\mathbf{H}$ depends on the chosen scaling of local time. 
For our purposes, it will be convenient to choose a scaling of local time such that $\E^0[\mathsf{L}^{-1}_1] = 1$ and hence, by Wald's equality, $\E^0[H_1] = \E^0[X_1]\E^0[\mathsf L^{-1}_1] = \E^0[X_1]$.

When upward regularity does not hold, $(\mathsf{L}^{-1}_t)_{t \geq 0}$ will not be strictly increasing and consequently, the ascending ladder height process $\mathbf H$ is a (possibly killed) compound  Poisson subordinator. In any scenario, the Laplace exponent $\Phi_H(\lambda)$, given by 
$$\Phi_H(\lambda)  = q + d_H \lambda + \int_0^\infty (1-\mathrm{e}^{-\lambda x})\, \Pi_H(\diff{x}), \quad \lambda \geq 0,$$
satisfies $\E^0[\exp(-\lambda H_t)] = \exp(-t\Phi_H(\lambda))$ and we refer to $d_H$ as the drift and $\Pi_H$ as the L\'evy measure of $\mathbf H$.

In the same vein, we can construct the ascending ladder height process $\hat{\mathbf H} = (\hat{H}_t)_{t \geq 0}$ for the dual L\'evy process $\hat \X = - \X$, which corresponds to time changing $\X$ by the right continuous inverse of local time $\hat{\mathsf{L}}$ at the \textit{infimum} of $\X$. Therefore, $\hat{\mathbf H}$ is referred to as the \textit{descending} ladder height process. If we denote by $\hat{\Phi}_H$ the Laplace exponent of $\hat{\mathbf H}$, then the Wiener--Hopf factorization tells us that $\X$ is fully characterized by means of the ascending and descending ladder height processes since the L\'evy--Khintchine exponent of $\X$ can be expressed as a factorization of the Laplace exponents of $\mathbf H$ and $\hat{\mathbf H}$, 
\begin{equation}\label{eq:wh}
\Psi(\theta) = -\mathsf{c} \Phi_H(-\mathrm{i}\theta)\hat{\Phi}_H(\mathrm{i}\theta), \quad \theta \in \R,
\end{equation}
where the constant $\mathsf{c}$ depends on the scaling of local time at the supremum and infimum. Among others, this factorization allows to express the characteristics of $\mathbf H$ in terms of the characteristics of $\X$ and $\hat{\mathbf H}$. A particularly useful identity for understanding the ascending ladder height L\'evy measure, usually referred to as Vigon's \'equation amicale invers\'e, was demonstrated in \cite{vigon2002}:
\begin{equation}\label{eq:vigon}
\Pi_H(\diff{x}) = \int_0^\infty \Pi(y+ \diff{x})\, \hat{U}_H(\diff{y}), \quad x >  0.
\end{equation}
Here, $\hat{U}_H(\diff{x}) = \E^0[\int_0^\infty \one_{\{\hat{H}_t \in \diff{x}\}} \diff{t}]$, $x \geq 0$, denotes the potential measure of $\hat{\mathbf H}$ and without loss of generality, the constant $\mathsf{c}$ in \eqref{eq:wh} is set to unity. 

While the theoretical solution strategy of the L\'evy driven impulse control problem is driven by the generator functional $f = \mathcal{A}_H \gamma$ with $\mathcal{A}_H$ denoting the generator of the ascending ladder height process $\mathbf H$, our data-driven strategy makes use of the link between $\mathbf H$ and \textit{overshoots} $\bm \cO = (\cO_t)_{t \geq 0}$ of $\X$ to find an estimator $\hat{f}_T$, from which the optimal reflection boundary $\coup^\ast$ is approximated by the greedy strategy \eqref{eq:strategy levy}. Let us first discuss classical properties of overshoots and then make the transition to the approach undertaken in \cite{doering21}, whose results are fundamental for the approximation properties of our statistical procedure. 

Let $t \geq 0$ be a given \textit{level} and  consider the overshoot $\cO_t$ over $t$, given by 
$$\cO_t = X_{T_t} - t$$
on $\{T_t < \infty\}$, where $T_t \coloneqq \inf\{s \geq 0: X_s > t\}$. Let us assume from here on that \ref{ass: upward} is  in place, which in particular implies $T_t < \infty$ almost surely, $0 < \E^0[H_1] < \infty$ and that $\mathbf H$ is an unkilled, strictly increasing subordinator. The fundamental link between $\bm\cO$ and $\mathbf H$ now stems from the observation that, due to its construction, the range of $\mathbf{H}$ almost surely coincides with the range of the running supremum process $(\overbar{X}_t)_{t \geq 0}$. As a consequence, the overshoot process $\bm\cO^H$ associated to $\mathbf H$ is indistinguishable from the overshoot process $\bm\cO$ associated to $\X$. Hence, if we want to estimate the characteristics of $\mathbf{H}$ (which cannot be observed based on a sample of $\X$ due to a lack of explicitness of the local time $\mathsf{L}$), one could hope for utilizing the overshoot link to get hold of $\mathbf H$ based on observations of $\X$, provided that $\bm \cO$ has some kind of regularity structures. 
Indeed, making use of the compensation formula (cf.\ \cite[Theorem 4.4]{kyprianou2014}), it can be shown that the law of $\cO_t = \cO^H_t$ is given by
$$\PP^x(\cO_t \in \diff{y}) = \delta_{x-t}(\diff{y})\one_{[0,x]}(t) + \int_{[0,t-x]} \Pi_H(u + \diff{y}) \, U_H(t-x -\diff{u}) \one_{(x,\infty)}(t), \quad y > 0,$$
(see \cite[Theorem 5.6]{kyprianou2014}) and resorting to classical renewal arguments it can be deduced from this formula that (cf.\ \cite[Theorem 5.7]{kyprianou2014})
\begin{equation} \label{eq: weak overshoot}
\PP^x(\cO_t \in \diff{y}) \overset{\mathrm{w}}{\underset{t \to \infty}{\longrightarrow}} \frac{1}{\E^0[H_1]} \big( d_H \delta_0(\diff{y}) + \one_{(0,\infty)}(y) \Pi_H((y,\infty))\big) \eqqcolon \mu(\diff{y}), \quad y \geq 0.
\end{equation}
Similar results can be obtained for overshoots associated to \textit{Markov additive processes} (MAPs), which are a natural generalization of L\'evy processes in the sense that these regime switching processes are constructed from a finite family of L\'evy processes and a continuous-time Markov chain, which governs the behaviour of the MAP according to the L\'evy process associated to the current state of the chain. By considering a MAP with a trivial underlying Markov chain, we obtain a L\'evy process by projection and hence results on overshoots of MAPs have direct analogues for overshoots of L\'evy processes. Let us therefore reformulate the findings from \cite{doering21}, where the authors study the convergence rate for stronger modes of convergence of MAP overshoots and consequently deduce the mixing behaviour of these objects, to the current setting.

The starting point for improving \eqref{eq: weak overshoot} in \cite{doering21} is the fact that, under the given assumptions, $\bm \cO$ can be shown to be a Feller process, whose unique invariant distribution is given by $\mu$ \cite[Proposition 3.3, Theorem 3.8]{doering21}.
This opens the door to tackling stability of $\bm\cO$ with the Meyn and Tweedie approach driving the complete statistical setting in this paper, as described in Section \ref{sec:2}. As a first crucial step, it is shown in \cite{doering21} that under some additional mild constraints on the characteristics of $\mathbf{H}$, which can be expressed in terms of constraints on the characteristics of $\X$, we can improve \eqref{eq: weak overshoot} to total variation convergence---or said differently to ergodicity of $\bm \cO$.

\begin{proposition}[Theorem 3.18 and Lemma 4.7 in \cite{doering21} and Theorem 7.11 in \cite{kyprianou2014}] \label{prop: ov tv}
Given \ref{ass: ergodic}, it holds that, for any $x \geq 0$,
\begin{equation} \label{eq:tv overshoot}
\PP^x(\cO_t \in \cdot) \overset{\mathrm{TV}}{\underset{t \to \infty}{\longrightarrow}} \mu.
\end{equation}
Moreover, \ref{ass: ergodic} is fulfilled if one of the following conditions hold:
\begin{enumerate}[leftmargin=*,label=(\roman*),ref=(\roman*)]
\item
 $\exists (a,b) \subset \R_+ \text{ s.t.\ } \bm{\lambda}\vert_{(a,b)} \ll \Pi\vert_{(a,b)}$;
\item
 $\X$ has bounded variation with L\'evy--Khinthchine exponent 
$$\Psi(\theta) = \mathrm{i}\delta \theta + \int_{\R} \big(\mathrm{e}^{\mathrm{i}\theta x} -1 \big)\, \Pi(\diff{x}),$$ 
and $\delta > 0$; \label{cond:bounded  var}
\item
 $\X$ has a Gaussian component;\label{cond:gauss}
\item
 $\X$ has positive jumps, unbounded variation, no Gaussian component and its L\'evy measure $\Pi$ satisfies 
$$\int_0^1 \frac{x \Pi((x,\infty))}{\int_0^x \int_y^1 \Pi((-1,-u)) \diff{u}\diff{y}} \diff{x} < \infty.$$ \label{cond:unbounded var}
\end{enumerate}
\end{proposition}
\begin{remark}
The mutually exclusive conditions \ref{cond:bounded var}-\ref{cond:unbounded var} are necessary and sufficient criteria for $d_H > 0$.
\end{remark}
To establish exponential rates of convergence in \eqref{eq:tv overshoot} under the natural assumption that $H_1$ possesses an exponential moment, \cite{doering21} combine a Lyapunov-type drift criterion given in \cite[Theorem 5.2]{DownMeynTweedie1995} with an explicit calculation of the resolvent kernel $$\mathcal{R_\lambda}(x,\cdot) \coloneqq \int_0^\infty \lambda \mathrm{e}^{-\lambda t} \PP^x(\cO_t \in \cdot) \diff{t},$$ which has the interpretation as the transition kernel of the Markov chain obtained from sampling $\X$ at subsequent independent exponential times with rate $\lambda$. Building on exponential ergodicity it is then shown in \cite{doering21} that $\bm \cO$ is exponentially $\beta$-mixing for any initial distribution possessing an exponential moment, which in particular includes the stationary distribution $\mu$. 

\begin{proposition}[Theorem 3.21, Theorem 3.24 and Lemma 4.7 in \cite{doering21}] \label{prop: ov exp}
Grant \ref{ass: ergodic} and \ref{ass: exp}. Then, convergence in \eqref{eq:tv overshoot} takes place at exponential rate. 
More precisely, for any $\delta \in (0,1)$ there exists a constant $c(\delta) > 0$ such that 
$$\big\lVert \PP^x(\cO_t \in \cdot) - \mu \big\rVert_{\mathrm{TV}} \leq c(\delta) \mathcal{R}_\lambda \exp(\lambda \cdot)(x) \mathrm{e}^{-t/(2+\delta)},\quad t \geq 0.$$
Moreover, if $\eta$ is some distribution on $(\R_+,\mathcal{B}(\R_+))$ such that $\eta(\exp(\lambda \cdot)) < \infty$, then, if $\X$ is started in $\eta$, $\bm \cO$ is exponentially $\beta$-mixing with rate 
$$\beta_{\PP^\eta}(t) \leq 2\varrho(\eta,\lambda,\delta) \mathrm{e}^{-t/(2+\delta)},$$
where 
$$\varrho(\eta,\lambda,\delta) = c(\delta)\sup_{t \geq 0} \E^\eta\big[\mathcal{R}_\lambda \exp(\lambda \cdot)(\cO_t) \big] < \infty.$$
Finally, \ref{ass: exp} is satisfied if and only if 
$$\int_1^\infty \mathrm{e}^{\lambda x} \, \Pi(\diff{x}) < \infty.$$
\end{proposition}

\end{appendix}

\printbibliography
\end{document}